\documentclass[11pt]{amsart}
\usepackage[margin=1in]{geometry}
\usepackage{graphicx}
\usepackage{soul}
\usepackage{amsthm, amsmath, amssymb, bm, bbm}
\usepackage{tikz}
\usepackage{mathrsfs}

\usepackage[utf8]{inputenc}
\usepackage[T1]{fontenc}
\usepackage[textsize=scriptsize,backgroundcolor=orange!5]{todonotes}

\usepackage{hyperref}
\usepackage{url}

\usepackage{mathtools}

\usepackage[noabbrev,capitalize]{cleveref}
\crefname{equation}{}{}
\usepackage{xcolor}

\usepackage{xcolor,graphics, graphicx}
\usepackage{floatrow}
\usepackage{verbatim}
\newfloatcommand{capbtabbox}{table}[][\FBwidth]

\numberwithin{equation}{section}

\newtheorem{theorem}{Theorem}[section]
\newtheorem{proposition}[theorem]{Proposition}
\newtheorem{lemma}[theorem]{Lemma}

\newtheorem{corollary}[theorem]{Corollary}

\newtheorem*{question*}{Question}
\Crefname{question}{Question}{Questions}

\theoremstyle{definition}

\newtheorem{remark}[theorem]{Remark}

\newcommand{\norm}{|\cdot|}

\newcommand{\1}{\mathbbm{1}}

\DeclareMathOperator{\Sym}{Sym}

\DeclareMathOperator{\rad}{rad}
\DeclareMathOperator{\GL}{GL}

\renewcommand{\Re}{\operatorname{Re}}

\newcommand{\Mod}[1]{\ (\mathrm{mod}\ #1)}
\DeclareMathOperator{\dd}{\mathbf{d}}
\DeclareMathOperator{\ee}{\mathbf{e}}

\newcommand{\mc}{\mathcal}

\newcommand{\A}{\mathbb{A}}

\newcommand{\C}{\mathbb{C}}

\newcommand{\Q}{\mathbb{Q}}
\newcommand{\RR}{\mathbb{R}}
\newcommand{\R}{\mathbb{R}}

\newcommand{\N}{\mathbb{N}}

\newcommand{\Z}{\mathbb{Z}}

\renewcommand{\pmod}[1]{\ (\mathrm{mod}\ #1)}

\renewcommand{\subset}{\subseteq}
\renewcommand{\epsilon}{\varepsilon}

\title{Patterns of Primes in Joint Sato--Tate Distributions}

\author[Chentouf]{A. Anas Chentouf}
\address{Massachusetts Institute of Technology, Cambridge, MA 02139, USA}
\email{chentouf@mit.edu}

\author[Cossaboom]{Catherine H. Cossaboom}
\address{University of Virginia, Charlottesville, VA 22903, USA}
\email{qkb9us@virginia.edu}

\author[Goldberg]{Samuel E. Goldberg}
\address{University of Virginia, Charlottesville, VA 22903, USA}
\email{seg8st@virginia.edu}

\author[Miller]{Jack B. Miller}
\address{Yale University, New Haven, CT 06511, USA}
\email{jack.miller.jbm82@yale.edu}

\date{\today}

\allowdisplaybreaks

\begin{document}

\maketitle

\begin{abstract}
For $j=1,2$, let $f_j(z) = \sum_{n=1}^{\infty} a_{j}(n) e^{2\pi i nz}$ be a holomorphic, non-CM cuspidal newform of even weight $k_j \ge 2$ with trivial nebentypus. For each prime $p$, let $\theta_{j}(p)\in[0,\pi]$ be the angle such that $a_j(p) = 2p^{(k-1)/2} \cos \theta_{j}(p)$. The now-proven Sato--Tate conjecture states that the angles $(\theta_j(p))$ equidistribute with respect to the measure $d\mu_{\mathrm ST} = \frac{2}{\pi}\sin^2\theta\,d\theta$.
We show that, if $f_1$ is not a character twist of $f_2$, then for  subintervals $I_1,I_2 \subset [0,\pi]$, there exist infinitely many bounded gaps between the primes $p$ such that $\theta_1(p) \in I_1$ and $\theta_2(p) \in I_2$.
We also prove a common generalization of the bounded gaps with the Green--Tao theorem. 
\end{abstract}

\section{Introduction}\label{sec:introduction}

The twin prime conjecture states that there are infinitely many pairs of consecutive primes $p_n$ and $p_{n+1}$ which differ exactly by $2$. Although the twin prime conjecture remains unproven, partial progress has been made over the past several decades. A consequence of the prime number theorem is that $p_n \sim n\log p_n$, so the average gap $p_{n+1}-p_n$ is asymptotically $\log p_n$, i.e., unbounded. In 2005, Goldston, Pintz, and Y\i ld\i r\i m 
\cite{GPY} developed a sieve method that detects tuples of primes in short intervals, demonstrating that
\begin{equation}
\label{eqn:GPY}
\liminf_{n \to \infty} 
\frac{p_{n+1} - p_n}{\log p_n} = 0.
\end{equation}
In particular, the result \eqref{eqn:GPY} says that for all $\varepsilon>0$, there exist infinitely many $n$ such that $p_{n+1} - p_n \leq \varepsilon \log p_n$. That is, there exist infinitely many consecutive primes whose gap is an arbitrarily small multiple of the mean gap.
In 2013, Zhang \cite{Zhang} significantly improved upon the GPY argument, showing that 
\[
\liminf_{n \to \infty} 
(p_{n+1}-p_n)
\le 70{,}000{,}000.
\]

Later that year, Maynard brought the value of $70{,}000{,}000$ down to $600$ \cite{BoundedGaps}.  He also proved that if $m \in \N$, then
\begin{equation*}
\liminf_{n \to \infty} (p_{n+m} - p_n)
\ll
m^3 e^{4m}.
\end{equation*}
(Tao developed the underlying ideas independently, but arrived at slightly different conclusions.) Later that year, the Polymath Research Project \cite{Polymath8b} synthesized the arguments of Zhang, Maynard, and Tao, proving that 
\[
\liminf\limits_{n\to\infty} (p_{n+1} - p_n)  \leq 246.
\]
More generally, the existence of infinitely many bounded gaps has been established within special sets of primes of number-theoretic interest. See \cite{Beatty,ChebThorner, VW}, for example.

In this paper, we are interested in bounded gaps between primes related to the Sato--Tate conjecture. 
Let $f(z)=\sum_{n=1}^{\infty} \lambda_f(n)e^{2\pi i nz} \in S_{k}^{\rm new}(\Gamma_0(N))$ be a holomorphic cusp form without complex multiplication (CM) and with trivial nebentypus, level $N$, and even integral weight $k \ge 2$. We normalize $f$ so that $\lambda_f(1) = 1$, and the trivial nebentypus condition then implies that $\lambda_f(n)\in\R$ for all $n$.
If $f$ is also an eigenform for all of the Hecke operators and all of the Atkin--Lehner involutions $|_kW(N)$ and $|_kW(Q_p)$ for each prime $p\mid N$, then $f$ is a newform (see \cite[Section 2.5]{OnoWeb}).
For any prime $p$, Deligne's proof of the Weil conjectures implies that $|\lambda_f(p)| \le 2p^{(k-1)/2}$. 
Thus, there exists a unique angle $\theta_f(p)\in[0,\pi]$ such that $\lambda_f(p)= {2p^{(k-1)/2}} \cos{\theta_f(p)}$.

It is natural to ask how the angles $\theta_f(p)$ distribute as the prime $p$ varies.
Given an interval $I \subset [0,\pi]$, we consider the set of primes 
\[
\mathcal{P}_{f,I}
:= 
\{p :  p \nmid N, \theta_f(p) \in I\}.
\]
One may ask how many primes $p\leq x$ belong to $\mathcal{P}_{f,I}$. Specifically, we are interested in the prime counting function
\[
\pi_{f,I}(x) 
:= 
\#(\mathcal{P}_{f,I} \cap [1,x]).
\]
The Sato--Tate conjecture, in the form proposed by Serre \cite{SerreST} which holds for all non-CM newforms as above, is equivalent to the statement that for the measure $d\mu_{\rm ST} := \frac{2}{\pi} \sin^2 \theta \,d\theta$,
\begin{equation}
\label{eqn:Sato Tate 1D limit}
\lim_{x\to\infty} \frac{\pi_{f,I}(x)}{\pi(x)} = 
\mu_{\rm ST}(I).
\end{equation}
This was proven by Barnet-Lamb, Geraghty, Harris, and Taylor \cite{ST}.
See the works of Thorner \cite{EffThorner} and Hoey, Iskander, Jin, and Trejos-Suárez \cite{HAAC} for effective bounds on the rate of convergence in \eqref{eqn:Sato Tate 1D limit}.

Given a pair of distinct non-CM holomorphic newforms $f_j(z) = \sum_{n=1}^{\infty} \lambda_j(n) e^{2\pi inz} \in S_{k_j}^{\textrm{new}}(\Gamma_0(N_j))$, $j = 1,2$, with $\lambda_{1}(1) = \lambda_{2}(1) = 1$, it is natural to ask what is the joint distribution of the angles $\theta_{1}(p)$ and $\theta_{2}(p)$. For a pair of subintervals $I_1,I_2 \subseteq [0,\pi ] $, we consider the set of primes 
\begin{equation}
\label{def:jointprimeset}
\mathcal{P}_{f_1,f_2,I_1,I_2} 
:= 
\{p :  p\nmid N_1 N_2, \ \theta_{1}(p) \in I_1, \ \theta_{2}(p) \in I_2 \}.
\end{equation}
Next, we define the joint prime counting function
\[
\pi_{f_1,f_2,I_1,I_2}(x) 
:= 
\#(\mathcal{P}_{f_1,f_2,I_1,I_2} \cap [1,x]).
\]
A natural question to ask, first posed by Mazur and Katz, is 
whether
\begin{equation}
\label{eqn:Sato Tate 2D limit}
\lim_{x\to\infty} \frac{\pi_{f_1,f_2,I_1,I_2}(x)}{\pi(x)} = 
\mu_{\rm ST}(I_1) \mu_{\rm ST}(I_2).
\end{equation}
That is, when do the angles $\theta_{1}(p)$ and $\theta_{2}(p)$ behave asymptotically like independent random variables?

We impose a technical condition on the relationship between $f_1$ and $f_2$ in order to obtain the asymptotic \eqref{eqn:Sato Tate 2D limit}. We say that $f_1$ and $f_2$ are \emph{twist-inequivalent}, denoted by $f_1 \not \sim f_2$, if for every primitive Dirichlet character $\chi$ one has that $f_1 \neq f_2\otimes\chi$.
Building on the work of Harris \cite{Harris}, Wong \cite{Wong} proved that \eqref{eqn:Sato Tate 2D limit} holds when $f_1\not\sim f_2$.  See Thorner \cite{EffThorner} for an effective bound on the rate of convergence.  In this paper, we establish the infinitude of bounded gaps within the set of primes $\mathcal{P}_{f_1,f_2,I_1,I_2}$ defined in \eqref{def:jointprimeset}.
\begin{theorem}
\label{thm:Maynard Tao}
For $j=1,2$, let $f_j(z) \in S_{k_j}^{\rm new}(\Gamma_0(N_j))$ be non-CM holomorphic cuspidal newforms with trivial nebentypus that are unrelated by twist. If $I_1 , I_2 \subseteq [0,\pi]$ are subintervals and $P_n$ denotes the $n$-th prime in $\mathcal{P}_{f_1,f_2,I_1,I_2}$, then
\begin{equation*}
\liminf_{n\to \infty} 
(P_{n+m} - P_{n}) 
\ll 
\frac{N_1N_2 m}{\varphi(N_1N_2) (\mu_{\rm ST}(I_1)\mu_{\rm ST}(I_2))^{8/3}}
\exp\Big(\frac{4581 N_1N_2m}{\varphi(N_1N_2)(\mu_{\rm ST}(I_1)\mu_{\rm ST}(I_2))^{8/3}}\Big).
\end{equation*}
The implied constant is absolute and effective.
\end{theorem}

In 2004, Green and Tao \cite{GreenTao} proved that the primes form infinitely many arithmetic progressions of arbitrary length. This was achieved by extending Szemer\'{e}di's theorem to the primes via a transference principle. Subsequently, many simplifications and improvements have been made to this argument, including a criterion of Pintz \cite{PintzCriterion}. 

Since $\mathcal{P}_{f_1,f_2,I_1,I_2}$ has positive density in the primes, the Green--Tao theorem implies the existence of infinitely many arithmetic progressions of arbitrarily length in $\mathcal{P}_{f_1,f_2,I_1,I_2}$. In this paper, we prove a common generalization of infinitely many bounded gaps with the Green--Tao theorem.
That is, there exist infinitely many arithmetic progressions of arbitrary length in $\mathcal{P}_{f_1,f_2,I_1,I_2}$ which occur simultaneously as the shifts of an admissible set. We postpone the full statement of this result until \cref{GreenTaothm} in \cref{sec:overview}.

\section*{Acknowledgements}

The authors were participants in the 2023 UVA REU in Number Theory. They are grateful for support from Jane Street Capital, the National Science Foundation (DMS-2002265 and DMS- 2147273), the National Security Agency (H98230-23-1-0016), the Massachusetts Institute of Technology, and the Templeton World Charity Foundation. The authors would like to thank Jesse Thorner for advising this project and for many helpful conversations, as well as Ken Ono for his valuable comments.

\section{Conventions and Notation}

Let $\mathbb{N}$ denote the set of positive integers. Let $\mathbb{P}$ denote the set of prime numbers $\{2,3,5, \ldots\}$. Throughout the paper, $p$ always denotes a prime. Given a set $S$, we let $\1_{S}$ denote the indicator function of $S$.

We let $a \mid b$ denote the proposition that $a$ divides $b$, and $a \nmid b$ to be its negation.
Although the notation $[a,b]$ is also used for intervals in the paper, we also use $[a,b]$ to denote the least common multiple of $a,b\in\N$. Similarly, we will also use $(a,b)$ to denote the greatest common divisor. In both cases, it will be clear from context which of the two meanings is intended.
We say that a finite subset $\mathcal{H}\subset\Z$ is admissible if for every prime $p$, the reduction of $\mathcal{H} \pmod{p}$ is a proper subset of $\Z/p\Z$. The diameter of an admissible set $\mathcal{H}$ is ${\rm diam}(\mathcal{H}) := \max_{1\leq i,j\leq k} |h_j-h_j|$.
We assume without loss of generality in our results that $0\in\mathcal{H}$.

Given a natural number $n \in \mathbb{N}$, it admits a prime factorization $n=p_1^{\alpha_1} \cdots p_r^{\alpha_r}$ where $p_1< \cdots < p_r$. We then define the following arithmetic functions:

\begin{itemize}
    \item $P^{-}(n):=p_1$, is the smallest prime divisor of $n$;
    \item $\rad(n):=p_1 \cdots p_r$, the product of the distinct primes dividing $n$;
    \item $\Lambda(n)$ is the von Mangoldt function, defined as $\Lambda(n) = \log{p}$ if $n=p^{\ell}$ for some prime $p$ and $\ell\in\N$, and $\Lambda(n) = 0$ otherwise;
    \item $d_k(n):=\prod_{j=1}^{r} \binom{k+\alpha_j-1}{k-1}$ is the $k$-fold divisor function; i.e., the number of ways $n$ can be expressed as the product of $k$ natural numbers;

    \item $\varphi(n):= n \prod_{j=1}^{r} (1-1/p_j)$ is the Euler totient function;

    \item $\mu(n)$ is the M\"{o}bius function, defined as $\mu(n)=(-1)^r$ if $n$ is squarefree and $\mu(n) = 0$ otherwise.
\end{itemize}

We write $f \ll g$, or $f = O(g)$, to indicate that in a range of $z$ (clear from context), there exists a positive constant $C$ that depends only on $k$, $\theta$, $\delta $, ${\rm diam}(\mathcal{H})$, $f_1$, $f_2$, $I_1$, $I_2$, $A$, and $B$ such that $|f(z)| \le C|g(z)|$.
Also, the use of the term ``sufficiently large'' may depend on $k$, $\theta$, $\delta $, ${\rm diam}(\mathcal{H})$, $f_1$, $f_2$, $I_1$, $I_2$, $A$, and $B$. The dependence of the implied constant on any additional parameters $u_1, \ldots, u_n$ is denoted by the symbol $\ll_{u_1, \ldots, u_n}$.
We write $f\asymp g$ to denote that $f\ll g$ and $f \gg g$.

For $z\in\C$ let $e(z) := e^{2\pi i z}$. Given a locally integrable function $f:\C\to\C$, for $\sigma\in\R$ we define
\[
\int_{(\sigma)} f(s)\,ds
:= 
\lim_{T\to\infty} \int_{\sigma-iT}^{\sigma+iT} f(s)\,ds
\]
provided that the limit exists.

For every $m\in\N\cup\{0\}$, we define the $m$-th Chebyshev polynomial of the second kind $U_m(t) \in \Z[t]$ by the recurrence relation
\begin{equation*}
U_0(t) := 1,
\qquad
U_1(t) := 2t,
\qquad
U_{m+2}(t) := 
2tU_{m+1}(t) - U_{m}(t).
\end{equation*}
In particular, we have that the Chebyshev polynomials satisfy
\begin{equation*}
U_m(\cos\theta) = \sum_{j=0}^{m} e^{i(m-2j)\theta}, \qquad
\theta\in\R.
\end{equation*}

\section{A Minorizing Fourier Expansion}
\label{sec:Fourier Expansion}

In order to prove bounded gaps in the prime sets $\mathcal{P}_{f_1,f_2,I_1,I_2}$, one must be able to detect when $(\theta_{1}(p),\theta_{2}(p)) \in I_1 \times I_2$. We proceed as in Rouse and Thorner \cite[Lemma 3.1]{RT}.

\begin{lemma}
\label{lem:minorizing}
Let $I_1, I_2 \subseteq [0,\pi]$ be subintervals. Given $M_1, M_2 \in \N$, there exists a real trigonometric polynomial of the form
\begin{equation}
\label{eqn:trigpoly}
\mathcal{F}_{I_1, I_2, M_1, M_2} ( \theta_1 , \theta_2) 
= 
\sum_{0 \leq m_1 \leq M_1}
\sum_{0 \leq m_2 \leq M_2}
\widehat{\mathcal{F}}_{I_1,I_2,M_1,M_2} (m_1 , m_2) U_{m_1} (\cos \theta_1) U_{m_2} (\cos \theta_2) 
\end{equation}
with the following properties.
\begin{enumerate}
\item If $\theta_1,\theta_2 \in [0,\pi]$, then $\mathcal{F}_{I_1 , I_2, M_1, M_2} (\theta_1, \theta_2) \leq \1_{I_1\times I_2} (\theta_1, \theta_2)$.
\label{item:F first property}
\vspace{5pt}

\item The inequality
\[
|\widehat{\mathcal{F}}_{I_1 , I_2 , M_1 ,M_2}(0,0) - \mu_{{\rm ST}}(I_1) \mu_{{\rm ST}}(I_2)| 
 \leq 
8\Big( \frac{\mu_{\rm ST} (I_2)^{1/3}}{M_1 +1} + \frac{\mu_{\rm ST} (I_1)^{1/3}}{M_2 +1} + \frac{6}{(M_1 +1)(M_2 + 1)} \Big)
\]
holds.
\label{item:F second property}
\end{enumerate}
\end{lemma}
\begin{proof}
We begin with some notation. Let $I_1 = [a_1,b_1]$, $I_2 = [a_2, b_2]$, and $J_1 := [a_1/(2\pi),b_1/(2\pi)]$, $J_2 := [a_2/(2\pi),b_2/(2\pi)]$.
Given an interval $J \subset [0,1]$, the Fourier expansion of the indicator function $\1_{J}$ is given by
\[
\1_{J}(\theta) = 
\sum_{n\in\Z} \widehat{\1}_{J}(n) e(n\theta),
\qquad 
\widehat{\1}_{J}(n) = \int_{J} e(-nt) \,dt.
\]
Cochrane \cite{Cochrane} proved that for $M_1,M_2\in\N$, there exists a trigonometric polynomial
\[
S(x_1,x_2) = 
\sum_{\substack{0 \leq |m_1 | \leq M_1 \\ 0 \leq |m_2 | \leq M_2}} \widehat{S} ( m_1 , m_2) \cdot e(m_1 x_1 + m_2 x_2)
\]
with the following properties:
\begin{enumerate}
\item for all $(x_1 , x_2) \in [ 0,1]^2 $, $S (x_1 , x_2 ) \leq \1_{J_1\times J_2} (x_1, x_2)$;
\vspace{5pt}

\item for all $(m_1 , m_2)$, we have that $\widehat{S} (m_1 , m_2) + \widehat{S} (- m_1 , - m_2) = 2 \Re \widehat{S} (m_1 , m_2)$;
\vspace{5pt}

\item for all $(m_1,m_2)$,
we have that
\begin{align}
\Big|\widehat{\1}_{J_1}(m_1) \widehat{\1}_{J_2}(m_2) - \widehat{S} (m_1 , m_2)\Big| 
 \leq 
\frac{|J_2|}{M_1 +1} + \frac{|J_1|}{M_2 +1} + \frac{3}{(M_1+1)(M_2+1)}.
\end{align}
\label{item:S third property}
\end{enumerate}
If we define $U_{-1} ( \cos \theta) = 0$ and $U_{-2} ( \cos \theta) = -1$, then
\begin{align}
\label{eqn:chebyconv}
\cos ( n \theta) 
= 
\frac{1}{2} ( U_n (\cos \theta) - U_{n-2} ( \cos \theta)), \hspace{10mm} n \geq 0.
\end{align}
We use these identities to obtain the Chebyshev polynomial approximation
\begin{align*}
\mathcal{F}_{I_1,I_2,M_1,M_2}(\theta_1 , \theta_2) 
= 
S\Big( \frac{\theta_1}{2\pi} , \frac{\theta_2}{ 2\pi} \Big) + 
S\Big( - \frac{\theta_1}{2\pi} , \frac{\theta_2}{ 2\pi} \Big) + 
S\Big( \frac{\theta_1}{2\pi} , - \frac{\theta_2}{ 2\pi} \Big) + 
S\Big(- \frac{\theta_1}{2\pi} , - \frac{\theta_2}{ 2\pi} \Big)
\end{align*}
for  $\1_{I_1\times I_2} ( \theta_1 , \theta_2)$.  Notice that $\mathcal{F}_{I_1,I_2,M_1,M_2}$ is real-valued. We see for $(\theta_1 , \theta_2) \in (0,\pi)^2 $ that
\[
\Big(- \frac{\theta_1}{2 \pi} ,  \frac{\theta_2}{2\pi} \Big) , \Big( \frac{\theta_1}{2 \pi} , - \frac{\theta_2}{2\pi} \Big) ,\Big(- \frac{\theta_1}{2 \pi} , - \frac{\theta_2}{2\pi} \Big)  \notin (J_1 \times J_2) \Mod{1}.
\]
Therefore, 
\begin{equation}
\label{eqn:polyineq1}
\mathcal{F}_{I_1,I_2,M_1,M_2} (\theta_1 , \theta_2) 
 \leq 
\1_{J_1\times J_2} \Big ( \frac{\theta_1}{2\pi}, \frac{\theta_2}{2\pi}\Big)
= 
\1_{I_1\times I_2} ( \theta_1 , \theta_2) .
\end{equation}
By continuity of $\mathcal{F}_{I_1,I_2,M_1,M_2}$, the inequality \cref{eqn:polyineq1} holds for all $(\theta_1 , \theta_2) \in [0,\pi ]^2 $.

We now write $\mathcal{F}_{I_1,I_2,M_1,M_2}$ in terms of the Chebyshev polynomials. First, we observe that
\[
\mathcal{F}_{I_1,I_2,M_1,M_2}( \theta_1 , \theta_2) = 4 \widehat{S} (0,0) +  4 \sum_{\substack{|m_1 | \leq M_1 \\ |m_2 | \leq M_2 \\ \mathbf{m} \neq \mathbf{0}}} \widehat{S} ( m_1 , m_2)  \cos ( |m_1 | \theta_1)  \cos (| m_2 | \theta_2).
\]
By \cref{eqn:chebyconv}, this equals
\begin{align*}
4& \widehat{S} (0,0) + \sum_{\substack{|m_1 | \leq M_1 \\ |m_2 | \leq M_2 \\ \mathbf{m} \neq \mathbf{0}}} \widehat{S} ( m_1 , m_2) \big( U_{|m_1|} (\cos \theta_1) - U_{|m_1| - 2} (\cos \theta_1) \big) \big( U_{|m_2|} (\cos \theta_2) - U_{|m_2| - 2} (\cos \theta_2) \big) .
\end{align*}
We compute
\[
\widehat{\mathcal{F}}_{I_1,I_2,M_1,M_2} (0,0) 
= 
4 \Re ( \widehat{S} (0,0) - \widehat{S} (2,0) - \widehat{S} (0,2) + \widehat{S} (2,2)) 
\]
and
\[
\mu_{\rm ST} ([a,b]) = 2\Re\Big( \widehat{\1}_{[\frac{a}{2\pi},\frac{b}{2\pi}]} (0) - \widehat{\1}_{[\frac{a}{2\pi},\frac{b}{2\pi}]}(2) \Big).
\]
Therefore, it follows from \cref{item:S third property} for $S$ that
\begin{align*}
| \widehat{\mathcal{F}}_{I_1,I_2,M_1,M_2} (0,0) - \mu_{\rm ST} (I_1) \mu_{\rm ST} (I_2) | &= 4 | \Re (  \widehat{S} (0,0) - \widehat{\1}_{J_1}(0) \widehat{\1}_{J_2}(0)  - \widehat{S} (2,0) + \widehat{\1}_{J_1}(2) \widehat{\1}_{J_2}(0) \\
&- \widehat{S} (0,2) + \widehat{\1}_{J_1}(0) \widehat{\1}_{J_2}(2) + \widehat{S} (2,2) - \widehat{\1}_{J_1}(2) \widehat{\1}(2)) | \\
& \leq 16 \Big( \frac{|J_2|}{M_1 +1} + \frac{|J_1|}{M_2 +1} + \frac{3}{(M_1 +1)(M_2 + 1)} \Big) .
\end{align*}

We claim that if $j \in \{ 1 ,2 \}$, then $|I_j| \leq \pi \mu_{\rm ST}(I_j)^{1/3}$. Indeed, writing $I_j = [x,x+y] \subseteq [ 0,\pi ]$, we compute
\[
\pi\mu_{\rm ST} ([x,x+y]) = 
2\int_x^{x+y} \sin^2 \theta \, d\theta 
= 
y - \sin (x+y) \cos (x+y) + \sin (x) \cos (x).
\]
Taking first and second partial derivatives with respect to $x$ for $0 \leq x \leq \pi -y$, one finds the function on the right is minimized (for fixed $y$) when $x=0$ (or $x=\pi -y $). So without loss of generality, we may let $x=0$. Now, it suffices to show
\[
V(y) := 
2\pi^2\int_{0}^{y} \sin^2(\theta)\,d\theta - y^3  \geq 0,
\qquad 
0  \leq y  \leq \pi.
\]
The critical points of $V(y)$ are $y=0,\pi,2.148\ldots$ and one computes $V(0) = V(\pi) = 0$ and $V(2.148\ldots) = 15.80\ldots$, therefore $V(y)\geq0$ in the interval $[0,\pi]$. Since $|J_i| = |I_i| / (2\pi)$ and $|I_i| \leq \pi \mu_{\rm ST} (I_i)^{1/3}$, \cref{item:F second property} follows for $\mathcal{F}_{I_1,I_2,M_1,M_2}$.
\end{proof}

Given subintervals $I_1,I_2\subseteq[0,\pi]$, we choose
\begin{equation}
\label{eqn:m1m2}
M_1 =  
\big\lceil 
27.8 \cdot \mu_{\rm ST}(I_1)^{-1} \mu_{\rm ST}(I_2)^{-2/3} - 1
\big\rceil , \qquad
M_2 =
\big\lceil 
27.8 \cdot \mu_{\rm ST}(I_1)^{-2/3} \mu_{\rm ST}(I_2)^{-1} -1
\big\rceil.
\end{equation}
and define
\begin{equation}
\label{eqn:minorizingpolydef}
F_{I_1,I_2} (\theta_1 , \theta_2) = \mathcal{F}_{I_1,I_2,M_1,M_2} (\theta_1 , \theta_2).
\end{equation}

\begin{corollary}
\label{cor:minorizing}
In the notation of \cref{eqn:m1m2} and \cref{eqn:minorizingpolydef}, the following properties hold:
\begin{enumerate}
    \item if $(\theta_1,\theta_2)\in[0,\pi]^2$, then $F_{I_1,I_2}(\theta_1,\theta_2)\leq \1_{I_1\times I_2}(\theta_1,\theta_2)$;
    \item $\widehat{F}_{I_1,I_2}(0,0) > 0$; and
    \item $(2(M_1+1)(M_2+1)+1)/\widehat{F}_{I_1,I_2}(0,0)
 \leq 
4581 (\mu_{\rm ST}(I_1)\mu_{\rm ST}(I_2))^{-8/3}$.
\end{enumerate}
\end{corollary}
\begin{proof}
This follows from \cref{item:F second property} in \cref{lem:minorizing}.
\end{proof}

We now state our main technical result, a Bombieri--Vinogradov type theorem for trigonometric approximations, \cref{thm:trigBVerror}, whose proof we postpone until \cref{sec:BV}. We use \cref{thm:trigBVerror} and \cref{cor:minorizing} in \cref{sec:overview} to prove \cref{thm:Maynard Tao}. 

\begin{theorem}
\label{thm:trigBVerror}
Let $f_1 \not \sim f_2 $ be non-CM holomorphic newforms with trivial nebentypus and levels $N_1,N_2$. Let $I_1 , I_2 \subseteq [0,\pi ]$ be subintervals of positive length. Let $\mathcal{F} = \mathcal{F}_{I_1,I_2,M_1,M_2}$ be a trigonometric polynomial as in \eqref{eqn:trigpoly}. Then, for all $\theta \in \Big(0,\frac{2}{2(M_1+1)(M_2+1)+1}\Big)$ and $A>0$,
\[
\sum_{\substack{q \leq x^{\theta} \\ (q,N_1N_2) = 1}} \max_{(a,q) = 1} 
\Big| \sum_{\substack{x < p \leq 2x \\ p \equiv a \Mod{q}}} \mathcal{F}  ( \theta_1(p) , \theta_2(p)) - \widehat{\mathcal{F}}(0,0) \frac{ \pi (2x) - \pi (x)}{\varphi (q)} \Big|
 \ll 
\frac{x}{(\log x)^A},
\]
where the implied constant only depends on $f_1,f_2,M_1,M_2,\theta,A$.
\end{theorem}

\section{Overview of Proof}
\label{sec:overview}

In this section, we give an overview of the proof of \cref{thm:Maynard Tao} as well as its common generalization with the Green--Tao theorem.

\subsection{Modified Maynard--Tao Sieve}

We present the necessary sieve theory for our proofs, following Maynard \cite{BoundedGaps}, and Vatwani and Wong \cite{VW}.
Let $F_{I_1,I_2}$ be the minorizing polynomial given in \cref{cor:minorizing} with degrees $M_1,M_2$.
We define several parameters which may depend on a uniform choice of $x\geq 16$:
\begin{equation*}
\widetilde{\theta} :=  \frac{1}{(M_1+1)(M_2+1)+2}, \qquad
0 < \theta < \widetilde{\theta}, \qquad  \delta \in  (0,\theta/2),
\end{equation*}
\begin{equation*}
R :=  x^{\theta/2-\delta}, \qquad D_0 >  0,\qquad W := \prod_{p \leq D_0} p,\qquad U :=  \prod_{\substack{p\leq D_0 \\ p\nmid N_1N_2}} p.
\end{equation*}
Let $\mathcal{H} = \{h_1,\ldots,h_k\}$ be an admissible set. For all primes $p$ there exists an integer $a_p$ such that $a_p \not \equiv h_j \Mod{p}$ for all $1 \leq j \leq k$. By the Chinese Remainder theorem, we can choose $u_0 \Mod{U}$ such that $u_0 \equiv - a_p \mod p$ for all $p \mid U$. This way, we have that $u_0 + h_j \not \equiv 0 \Mod{p} $ for all $p \mid U$ and $1 \leq j \leq k $. Therefore, we deduce that $(u_0 + h_j , U) = 1$ for all $1 \leq j \leq k$.
For every $k$-tuple of natural numbers $d_1,\ldots,d_k$, we introduce a parameter
\[
\lambda_{d_1,\ldots,d_k} \in \R
\]
that satisfies the following support condition.
The parameter $\lambda_{d_1,\ldots,d_k}$ is nonzero only if $d := \prod_{j=1}^k d_j$ satisfies $d < R$, $\mu^2(d) = 1$, and we also have that $(d_j,W) = 1$ for all $1\leq i \leq k$. Notice that the condition $\mu^2(d) = 1$ implies that the $d_j$ are squarefree and pairwise coprime.
We define for each $n\in\N$ a nonnegative weight
\[
w_n := 
\begin{cases}
\Big(\sum\limits_{d_j \mid n+h_j \, \forall j} \lambda_{d_1,\ldots,d_k} \Big)^2 &\mbox{if $n \equiv u_0 \pmod{U}$,} \\
0 &\mbox{otherwise.}
\end{cases}
\]
We also define the sums 
\begin{align*}
S_1 &:= 
\sum_{\substack{x < n \leq 2x \\ n \equiv u_0 \Mod {U}}} w_n ,
\\[1em]
S_2 &:= \sum_{\substack{x < n \leq 2x \\ n \equiv u_0 \Mod {U}}} \Big( \sum_{j=1}^k \1_{\mathbb{P}} (n+h_j) F_{I_1,I_2} (\cos  \theta_1(n+h_j), \cos \theta_2(n+h_j)) \Big) w_n,
\\[1em]
S(x,\rho) &:= S_2 - \rho S_1.
\end{align*}
Although $\theta_{j}(p)$ is defined only for $p$ prime, by abuse of notation this is acceptable because the indicator function $\1_{\mathbb{P}}(n+h_j)$ is non-vanishing if and only if $n+h_j$ is prime. 
Note that because $F_{I_1,I_2}$ satisfies \cref{item:F first property}, we have that
\begin{align*}
\sum_{x < n \leq 2x} \Big( \sum_{j=1}^k \1_{\mathbb{P}} (n+h_j) \1_{I_1\times I_2} ( \theta_1(n+h_j) ,  \theta_2(n+h_j)) - \rho \Big) w_n
 \geq 
S (x,\rho).
\end{align*}
Thus, it suffices to show that $S (x, \rho) > 0 $ for all $x$ sufficiently large in order to deduce by the pigeonhole principle that there are infinitely many integers $n\in\N$ such that the set $\{n+h_1,\ldots,n+h_k\}$ contains at least $\lfloor \rho + 1 \rfloor$ primes which lie in $\mathcal{P}_{f_1,f_2,I_1,I_2}$.

\subsection{Bounded Gaps}

In this section, we prove \cref{thm:Maynard Tao} using three propositions that are analogues of \cite[Propositions 4.1, 4.2, 4.3]{BoundedGaps}. We postpone the proof of these propositions in \cref{sec: proof of propositions}. While \cref{thm:Maynard Tao} follows from the second and third propositions, we state the first proposition here as well for completeness.

The first proposition expresses the asymptotics for the sums $S_1$ and $S_2$ in terms of the integral operators of a Riemann-integrable function $G:[0,1]^k\to\R$, defined by
\begin{align*}
I_k(G) &:= 
\int_0^1 \cdots \int_0^1 G(t_1,\ldots,t_k)^2 \,dt_1 \cdots dt_k,
\\
J_k^{(m)}(G) &:= 
\int_{0}^{1} \cdots \int_{0}^{1} \Big(
\int_{0}^{1} G(t_1,\ldots,t_k) \,dt_m
\Big)^2 dt_1 \cdots dt_{m-1} dt_{m+1} \cdots d_k.
\end{align*}

\begin{proposition}
\label{prop:Maynard Prop S1 and S2}
Let $\theta \in \Big(0,\frac{1}{(M_1+1)(M_2+1)+2}\Big)$.
Let $D_0$ be sufficiently large.
Let $G: [0,1]^k \to \mathbb{R}$ be a smooth function supported on the simplex $\{ (t_1 , \ldots, t_k) \in [0,1]^k : \sum_{j=1}^k t_j \leq 1 \}$, and define
\[
\lambda_{d_1 ,\ldots,d_k} := \Big( \prod_{j=1}^k \mu (d_j) d_j \Big) \sum_{\substack{r_1 , \ldots,r_k \\ d_j \mid r_j \\ (r_j , W) =1}} \frac{\mu (\prod_{j=1}^k r_j)^2}{\prod_{j=1}^k \varphi (r_j)} G \Big( \frac{\log r_1}{\log R} ,\ldots, \frac{\log r_k}{\log R} \Big) .
\]
In particular, $\lambda_{d_1 , \ldots, d_k} = 0$ unless $\prod_{j=1}^k d_j$ is at most $R$, coprime to $W$, and squarefree. Then, we have that
\begin{align}
\label{eqn:Maynard S1}
S_1 &= \rad(N_1N_2) \cdot \frac{(1+O(1/D_0)) \varphi (W)^k x ( \log R)^k}{W^{k+1}} I_k (G), 
\\[1em]
\label{eqn:Maynard S2}
S_2 &=
\varphi(\rad(N_1N_2)) \cdot
\widehat{F}_{I_1,I_2} (0,0) \cdot 
\frac{(1+O(1/D_0))\varphi (W)^k x ( \log R)^{k+1}}{W^{k+1} \log x} \sum_{m=1}^k J_{k}^{(m)} (G).
\end{align}
\end{proposition}

\begin{remark}
\label{rem:o(1) is O(1/D0)}
In Maynard's work \cite{BoundedGaps}, he defines $D_0$ as $\log \log \log x$. It was noticed by Pintz \cite{AdditionalPintz} that \cref{prop:Maynard Prop S1 and S2} holds for any choice of $D_0$ sufficiently large. Indeed, Maynard's estimates \cite[Lemmata 6.2, 6.3]{BoundedGaps} are strong enough for this purpose.
\end{remark}

The second proposition allows us to convert the ratio of integral operators into a statement about the existence of primes in bounded intervals.

\begin{proposition}
\label{prop:Maynard ratio calculation}
Let $\mathcal{H} = \{ h_1 , \ldots, h_k \} $ be an admissible set and $\mathcal{S}_k$ denote the set of all Riemann-integrable real functions supported on $\{ ( t_1 , \ldots, t_k) \in [0,1]^k : \sum_{j=1}^k t_j \leq 1 \} $. Fix $\theta \in (0,\widetilde{\theta})$, and define
\begin{equation}
\label{eqn:Mk and rk def}
\mathcal{M}_k 
:= 
\sup_{G \in \mathcal{S}_k} \frac{\sum_{m=1}^k J_{k}^{(m)} (G)}{I_k (G)}, 
\hspace{10mm}
r_k := 
\Big\lceil
\frac{\varphi(N_1N_2)}{N_1N_2}
\widehat{F}_{I_1,I_2} (0,0) \mathcal{M}_k \frac{\theta} {2} \Big\rceil.
\end{equation}
Then, there are infinitely many integers $n$ such that at least $r_k$ of the numbers $n+h_j$ lie in $\mathcal{P}_{f_1,f_2,I_1,I_2}$. 
In particular, letting $P_{n}$ be the $n$-th prime in $\mathcal{P}_{f_1,f_2,I_1,I_2}$, we have that
\[
\liminf\limits_{n\to\infty} (P_{n+r_k-1} - P_{n}) 
 \leq 
{\rm diam}(\mathcal{H}).
\]
\end{proposition}

The third proposition, which was proven in Polymath Research Project 8b \cite[Theorem 3.9]{Polymath8b}, provides a lower bound on the value $\mathcal{M}_k$.

\begin{proposition}
\label{prop:polymath8b Mk}
There exists an absolute and effectively computable constant $c_1>0$ such that if $\mathcal{M}_k$ is the quantity defined in \eqref{eqn:Mk and rk def} and $k \ge c_1$, Then 
\[
\mathcal{M}_k \ge \log k - c_1 .
\]
\end{proposition}
\begin{proof}
See \cite[Theorem 3.9]{Polymath8b}.
\end{proof}

We now use \cref{prop:Maynard ratio calculation,prop:polymath8b Mk} to prove a more explicit version of \cref{thm:Maynard Tao}. The proof of \cref{thm:Maynard Tao} then follows.

\begin{theorem}
\label{thm:more explicit Maynard Tao}
Let $I_1, I_2 \subseteq [0, \pi]$ be proper subintervals with positive Lebesgue measure.
Letting $P_{n}$ be the $n$-th prime in $\mathcal{P}_{f_1,f_2,I_1,I_2}$, we have that for $m\in\N$,
\[
\liminf_{n \to \infty} (P_{n+m} - P_{n})
\ll
\frac{N_1N_2m}{\varphi(N_1N_2) \widehat{F}_{I_1,I_2}(0,0) \widetilde{\theta}} \exp \Big( \frac{2N_1N_2m}{\varphi(N_1N_2) \widehat{F}_{I_1,I_2} (0,0) \widetilde{\theta}} \Big),
\]
where the implied constant is absolute and effective.
\end{theorem}

\begin{proof}
We define
\[
k := 
\max \Big(3, \Big\lceil \exp \Big(\frac{2N_1N_2m}{\varphi(N_1N_2) \widehat{F}_{I_1,I_2} (0,0) \widetilde{\theta}}  + c_1 + 1 \Big) \Big\rceil \Big),\qquad \theta 
:= 
\Big(1 - \frac{1}{\log k}\Big) \widetilde{\theta}.
\]
Consider $r_k$ as in \eqref{eqn:Mk and rk def}. Using \cref{prop:polymath8b Mk}, we have that
\[
r_k \geq 
\frac{\varphi(N_1N_2)}{N_1N_2} \widehat{F}_{I_1,I_2} (0,0)(\log k - c_1)\Big( 1 - \frac{1}{\log k} \Big) \frac{\widetilde{\theta}}{2} >  \frac{\varphi(N_1N_2)}{N_1N_2}
\frac{\widetilde{\theta}}{2}\widehat{F}_{I_1,I_2} (0,0) ( \log k - c_1 - 1)\geq  m.
\]
Thus, by setting $\mathcal{H} := \{p_{\pi(k)+1},\ldots,p_{\pi(k)+k}\}$ we have that, by the prime number theorem,
\begin{align*}
{\rm diam}(\mathcal{H})  \ll  
k \log k  \ll   \frac{N_1N_2m}{\varphi(N_1N_2)\widehat{F}_{I_1,I_2} (0,0) \widetilde{\theta}} 
\exp\Big(\frac{2N_1N_2m}{\varphi(N_1N_2) \widehat{F}_{I_1,I_2} (0,0) \widetilde{\theta}}\Big),
\end{align*}
where the implied constants are absolute and effective.
By \cref{prop:Maynard ratio calculation}, the proof is complete.
\end{proof}

\begin{proof}[Proof of \cref{thm:Maynard Tao}]

This follows from \cref{thm:more explicit Maynard Tao} using \cref{cor:minorizing}.
\end{proof}

\subsection{Arithmetic Progressions}

For every positive constant $\gamma > 0$, we define the set of numbers
\[
\mathcal{R}(\mathcal{H},\gamma) 
:= 
\bigg\{ n \in \N : P^{-}\Big(\prod_{j=1}^{k} (n+h_j) \Big) \geq n^{\gamma}\bigg\}.
\]

We begin by stating a theorem of Pintz \cite{PintzCriterion}, which provides sufficient criteria for the existence of arithmetic progressions of arbitrary length. 

\begin{theorem}[{\cite[Theorem~5]{PintzCriterion}}]\label{thm:Pintz}
Let $\mathcal{H} = \{h_1, \ldots, h_k\}$ be an admissible set. Given a set $\mathcal{A}\subset\Z$, suppose there exist constants $\gamma_1,\gamma_2>0$ such that
\begin{equation}
\label{eqn:Pintz criteria}
\mathcal{A} \subset \mathcal{R}(\mathcal{H},\gamma_1),
\qquad
\#\{n \leq x : n \in \mathcal{A}\} 
\ge
\gamma_2 \frac{x}{(\log x)^k} 
\,\emph{ for all}x \emph{ sufficiently large}.
\end{equation}
Then, $\mathcal{A}$ contains infinitely many arithmetic progressions of arbitrary length.
\end{theorem}

We have the following proposition whose proof we present in \cref{sec:Green Tao}.

\begin{proposition}
\label{prop:Pintz is satisfied}
Consider the set
\begin{equation}
\label{eqn:auxiliary set}
\mathcal{A}_{f_1,f_2,I_1,I_2}(\mathcal{H},\gamma_1) :=
\Big\{
n\in\N : 
\#\Big((n+\mathcal{H}) \cap \mathcal{P}_{f_1,f_2,I_1,I_2} \Big) \geq m+1
\Big\}
\cap \mathcal{R}(\mathcal{H},\gamma_1).
\end{equation}
There exists a constant $c_2 >0$ such that if $m\in\N$ satisfies
\begin{equation}
\label{eqn:m versus k}
\exp \Big(\frac{2N_1N_2m}{\varphi(N_1N_2) \widehat{F}_{I_1,I_2} (0,0) \widetilde{\theta}}\Big) 
 \leq 
c_2 k,
\end{equation}
then there exists $\gamma_1>0$ sufficiently small so that $\mathcal{A}_{f_1,f_2,I_1,I_2}(\mathcal{H},\gamma_1)$ contains infinitely many arithmetic progressions of arbitrary length.
\end{proposition}

Using \cref{thm:Pintz,prop:Pintz is satisfied}, we prove the following theorem.

\begin{theorem} 
\label{GreenTaothm}
Let $I_1, I_2$ be subintervals of $[0,\pi]$ with positive Lebesgue measure.
Let $\mathcal{H}=\{h_1, \ldots, h_k\}$ be an admissible set. 
If $m\in\N$ satisfies \eqref{eqn:m versus k}, then there exists an $(m+1)$-element subset $\mathcal{H}'\subset\mathcal{H}$ such that
\[
\{
n\in\N : n+\mathcal{H}' \subset \mathcal{P}_{f_1,f_2,I_1,I_2}
\}
\]
contains infinitely many arithmetic progressions of arbitrary length.
\end{theorem}

\begin{proof}
We follow the argument of \cite{STprimes}.
By \cref{thm:Pintz,prop:Pintz is satisfied}, the set $\mathcal{A}_{f_1,f_2,I_1,I_2}(\mathcal{H},\gamma_1)$ contains infinitely many $t$-term arithmetic progressions for each $t\in\N$.

Let $\mathfrak{Col}$ be the set of all $(m+1)$-subsets of $\mathcal{H}$.
Let $M$ be a positive integer we choose later. Because $\mathcal{A}_{f_1,f_2,I_1,I_2}(\mathcal{H},\gamma_1)$ contains infinitely many $M$-term arithmetic progressions, we let \\ $\{n_1 < \cdots < n_M\} \subset \mathcal{A}_{f_1,f_2,I_1,I_2}(\mathcal{H},\gamma_1)$ be one such progression.
Notice that for every $n_j$, there exists an $(m+1)$-subset $\mathcal{C}_j \in \mathfrak{Col}$ such that $n_j+\mathcal{C}_j \in \mathcal{P}_{f_1,f_2,I_1,I_2}$.

We view $\mathfrak{Col}$ as a finite set of colors and color the index $j$ with $\mathcal{C}_j$ for each $j = 1,\ldots,M$. 
By van der Waerden's theorem, for $M = M(t,m,k)$ sufficiently large, there exists a subset of indices $L \subset \{1,\ldots,M\}$ which form a $t$-term monochromatic progression. In particular, it follows that the set $\{n_j : j\in J\} \subset \mathcal{A}_{f_1,f_2,I_1,I_2}(\mathcal{H},\gamma_1)$ is a $t$-term monochromatic progression (coloring $n_j$ with $\mathcal{C}_j$ for each $j=1,\ldots,M$).
Because $\mathcal{A}_{f_1,f_2,I_1,I_2}(\mathcal{H},\gamma_1)$ has infinitely many $M$-term progressions, we deduce that $\mathcal{A}_{f_1,f_2,I_1,I_2}(\mathcal{H},\gamma_1)$ has infinitely many $t$-term monochromatic progressions.
Therefore, by the pigeonhole principle, there exists an $(m+1)$-subset $\mathcal{H'}\subset\mathcal{H}$ that occurs as a monochromatic coloring for a $t$-term progression infinitely many times.
Then, $\mathcal{H'}$ satisfies the theorem as desired.
\end{proof}

\subsection{Structure of the Paper}
The remainder of our paper is dedicated to two primary objectives. The first is to prove \cref{prop:Maynard Prop S1 and S2,prop:Maynard ratio calculation,prop:Pintz is satisfied}, which we do in \cref{sec: proof of propositions,sec:Green Tao}. The second is to prove \cref{thm:trigBVerror}, which we do in \cref{sec:Automorphic L-functions} through \cref{sec:BV}. \cref{sec:Automorphic L-functions,sec: symmetric powers} provide the necessary background on the theory of $L$-functions. We prove \cref{thm:trigBVerror} in \cref{sec:BV}.

\section{Proof of Maynard--Tao Propositions}
\label{sec: proof of propositions}

In this section we prove \cref{prop:Maynard Prop S1 and S2,prop:Maynard ratio calculation}.
Throughout, we often write $\dd$ and $\ee$ to abbreviate $(d_1 ,\ldots, d_k)$ and $(e_1, \ldots, e_k)$ respectively.

\begin{proof}[Proof of \cref{prop:Maynard Prop S1 and S2}]

To handle $S_1$, we expand the square of $w_n$ and swap the order of summation to obtain
\begin{equation}
\label{eqn:Maynard S1hat swap}
S_1 = 
\sum_{\dd, \ee} \lambda_{\dd} \lambda_{\ee}
\sum_{\substack{x < n \leq 2x \\ n\equiv u_0 \Mod{U} \\ [d_j,e_j] \mid n+h_j \forall j}} 1.
\end{equation}
We claim that the summands of $S_1$ are nonzero only if $W,[d_1,e_1],\ldots,[d_k,e_k]$ are pairwise coprime. This is because the only case that $W,[d_1,e_1],\ldots,[d_k,e_k]$ are not pairwise coprime is when $(d_j,e_{\ell}) > 1$ for some $j\neq \ell$.
However, $(d_j,e_{\ell})$ divides $h_j-h_{\ell}$ but is also coprime to $W$, and so taking $D_0$ sufficiently large so that $|h_j-h_{\ell}| \leq D_0$, we conclude that $(d_j,e_{\ell}) = 1$.

Because we may assume that $W,[d_1,e_1],\ldots,[d_k,e_k]$ are pairwise coprime and $U \mid W$, we may also assume that $U,[d_1,e_1],\ldots,[d_k,e_k]$ are pairwise coprime.
Our conditions for the index $n\in\N$ above are that $n\in(x,2x]$, $n \equiv u_0 \pmod{U}$, and $n+h_j \equiv 0 \pmod{[d_j,e_j]}$ for all $j$.
By applying the Chinese remainder theorem to simplify the $k+1$ congruence conditions that we have, we may replace them by the condition $n \equiv a \pmod{q}$, where $q := U \prod_{j=1}^k [d_j,e_j]$ and $a$ is some residue class modulo $q$.
Therefore, the inner sum in \eqref{eqn:Maynard S1hat swap} simplifies as $x/q + O(1)$, and we obtain

\[
S_1 
= 
\frac{x}{U} \sideset{}{'}{\sum}_{\dd, \ee} \frac{\lambda_{\dd} \lambda_{\ee}}{\prod_{j=1}^k [d_j,e_j]} + O\Big(\sideset{}{'}{\sum}_{\dd, \ee} |\lambda_{\dd} \lambda_{\ee}| \Big),
\]
where $\sideset{}{'}{\sum}$ denotes the restriction that $W,[d_1,e_1],\ldots,[d_k,e_k]$ are pairwise coprime. This gives the same sum as in \cite{BoundedGaps} by pulling out a factor of $W/U$. Note that $W/U = \rad(N_1N_2)$ for $D_0$ sufficiently large. 
Therefore, \eqref{eqn:Maynard S1} follows by the identical argument presented in \cite{BoundedGaps}, using \cref{rem:o(1) is O(1/D0)} to replace $o(1)$ with $O(1/D_0)$.

We now handle the case of $S_2$. For each $m \in \{1,\ldots,k\}$ let
\begin{align}
S_2^{(m)}
:\hspace{-3pt}&= 
\sum_{\substack{x < n \leq 2x \\ n \equiv u_0 \Mod {U}}} \1_{\mathbb{P}} (n+h_m) F_{I_1,I_2} ( \cos \theta_1(n+h_m) , \cos \theta_2(n+h_m)) \Big( \sum_{d_j \mid n+h_j \forall j} \lambda_{\dd} \Big)^2
\\[1em]
\label{eqn:Maynard S2hat swap}
&= \sum_{\dd , \ee} \lambda_{\dd} \lambda_{\ee} \sum_{\substack{x < n \leq 2x \\ n \equiv u_0 \Mod{U} \\ [d_j , e_j] \mid n+h_j}} \1_{\mathbb{P}} (n+h_m) F_{I_1,I_2} ( \cos \theta_1(n+h_m) , \cos \theta_2(n+h_m)).
\end{align}
In particular, we have that $S_2 = \sum_{m=1}^k S_2^{(m)}$.
As we have argued in the case of $S_1$, the summands of $S_2^{(m)}$ are nonzero only if $W,[d_1,e_1],\ldots,[d_k,e_k]$ are pairwise coprime, and the argument presented for $S_1$ still holds. Now, let us notice that for our inner sum in \eqref{eqn:Maynard S2hat swap} we may assume
\[
n+h_m \text{ is prime},
\qquad
[d_m,e_m] \mid n+h_m,
\qquad 
(d_me_m,W) = 1,
\qquad
\text{and}
\quad
d_me_m \leq R^2 = x^{\theta/2-\delta}.
\]
If $[d_m,e_m] \neq 1$, we have that $[d_m,e_m] = n+h_m > x+h_m$. On the other hand, $[d_m,e_m] \leq x^{\theta/2-\delta}$ and $\theta/2-\delta < 1$, we arrive at a contradiction for all $x = x(\mc{H},\theta,\delta)$ sufficiently large. Therefore, we may assume that $[d_m,e_m] = 1$, in particular $d_m = e_m = 1$.

Similar to our argument for $S_1$, we may use the Chinese remainder theorem applied to the $k+1$ congruence conditions modulo $U,[d_1,e_1],\ldots,[d_k,e_k]$ to write our inner sum in \eqref{eqn:Maynard S2hat swap} as taking place over $n\in(x,2x]$ and $n \equiv a \pmod{q}$, where $q := U \prod_{j=1}^k [d_j,e_j]$ and $a$ is some residue class modulo $q$. We may assume that $n+h_m$ is a prime for some $n\in(x,2x]$, otherwise the inner sum is zero.
That is, we may assume there exists a prime $p = x + O({\rm diam}(\mathcal{H}))$ that is congruent to $a+h_m \pmod{q}$.
Notice that $(a+h_m,q)$ divides $p$, and so either $(a+h_m,q)=1$, or
\[
x + O({\rm diam}(\mathcal{H}))
= 
p  \leq 
q
= O_{\delta}(x^{\theta/2}).
\]
By taking $x$ sufficiently large we obtain $p>q$, therefore $(a+h_m,q) = 1$.

We may write
\begin{equation}
\label{eqn:Maynard S2hat modulo q}
S_2^{(m)}
= 
\sum_{\dd , \ee} \lambda_{\dd} \lambda_{\ee} \sum_{\substack{x < n \leq 2x \\ n+h_m \equiv a \Mod{q}}} \1_{\mathbb{P}} (n+h_m) F_{I_1,I_2} ( \cos \theta_1(n+h_m) , \cos \theta_2(n+h_m)).
\end{equation}
Notice by shifting our index $n\in\N$ by $-h_m$ that the inner sum is equal to
\[
\sum_{\substack{x < n \leq 2x \\ n  \equiv a \Mod{q}}} \1_{\mathbb{P}} (n) F_{I_1,I_2} ( \cos \theta_1(n) ,  \cos \theta_2(n)) + O({\rm diam}(\mathcal{H})).
\]
Because the inner sum index is counting primes in residue classes, it is natural that we want to consider the error term
\[
E(x,q) := 1 + \max_{(a,q)  = 1} \Big| \sum_{\substack{ x < n \leq 2x \\ n  \equiv a \Mod{q}}} \1_{\mathbb{P}} (n) F_{I_1,I_2} ( \cos \theta_1(n) ,  \cos \theta_2(n)) - \widehat{F}_{I_1,I_2} (0,0)\frac{\pi (2x) - \pi (x)}{\varphi (q)}\Big|.
\]
By shifting indices and isolating a main term from the error term $E(x,q)$, we are able to write the inner sum in \eqref{eqn:Maynard S2hat modulo q} as
\begin{align}
\label{eqn:Maynard S2hat simplified} 
\widehat{F}_{I_1,I_2} (0,0) \frac{\pi (2x) - \pi (x)}{\varphi (q)} + O(E(x,q)).
\end{align}
Substituting \eqref{eqn:Maynard S2hat simplified} for the inner sum of \eqref{eqn:Maynard S2hat modulo q} and using the multiplicativity of $\varphi$ to pull out a factor of $\varphi(U)$ from $\varphi(q)$, we obtain
\[
S_2^{(m)}
= 
\widehat{F}_{I_1,I_2} (0,0)\frac{\pi (2x) - \pi (x)}{\varphi (U)} \sideset{}{'}{\sum}_{\substack{\dd , \ee \\ d_m = 1 \\ e_m = 1}}  \frac{\lambda_{\dd} \lambda_{\ee}}{\prod_{j=1}^k \varphi ( [ d_j ,e_j ])} + O \Big( \sideset{}{'}{\sum}_{\dd , \ee}  | \lambda_{\dd} \lambda_{\ee} |  E(x,q) \Big) ,
\]
where $\small\sideset{}{
'}{\sum}$ denotes the restriction that $W, [d_1 , e_1], \ldots, [d_k,e_k]$ are pairwise coprime.

We now employ \cref{thm:trigBVerror} to bound the error term.
We have as in \cite[(5.9) and (6.3)]{BoundedGaps} that $\lambda_{\max} \ll G_{\max} ( \log R)^k$, where
\[
\lambda_{\max} := 
\max_{d_1,\ldots,d_k} |\lambda_{d_1,\ldots,d_k}|
\]
and
\[
G_{\max} := 
\sup_{(t_1,\ldots,t_k)\in[0,1]^k} |G(t_1,\ldots,t_k)| + \sum_{j=1}^k \Big|
\frac{\partial G}{\partial t_j}(t_1,\ldots,t_k)
\Big|.
\]
This gives a bound for $| \lambda_{\dd} \lambda_{\ee} |$. By the support of $\lambda_{d_1,\ldots,d_k}$, we can restrict the range of possible $q$'s to squarefree $r < R^2 W$ such that $(r,N_1 N_2) = 1 $. Now, for any squarefree $r$, there are at most $\tau_{3k} (r)$ choices of $d_1 , \ldots, d_k , e_1 ,\ldots,e_k$ such that $r = U \prod_{j=1}^k [ d_j , e_j ]$ where $\tau_{3k} (r)$ is the number of ways to write $r$ as a product of $3k$ positive integers. Thus, we obtain an error term of
\begin{align*}
\ll  y_{\max}^2 ( \log R)^{2k} \sum_{\substack{r < R^2 W \\ (r,N_1 N_2) = 1}} \mu (r)^2 \tau_{3k} (r)  E(x,r) .
\end{align*}
Applying Cauchy-Schwarz, we obtain an error that is
\[
\ll  y_{\max}^2 ( \log R)^{2k} \Big( \sum_{r < R^2 W} \mu (r)^2 \tau_{3k}^2 (r) E(x,r) \Big)^{1/2} \Big( \sum_{\substack{r < R^2 W \\ (r,N_1 N_2) = 1}} \mu (r)^2 E(x,r) \Big)^{1/2} .
\]
Notice that for $E(x,r)$ we have that
\begin{equation*}
E(x,r) 
 \ll_{F_{I_1,I_2}}  
\max_{(a,q)=1} \frac{\pi(2x;r,a)-\pi(x;r,a)}{\varphi(r)} + \frac{\pi(2x)-\pi(x)}{\varphi(r)} \ll \frac{x}{\varphi(r)}.
\end{equation*}
Therefore, the middle factor is $O(x^{1/2} (\log x)^{3k/2})$. On the other hand, note that $R^2 W \leq x^{\theta} $ for $x = x(\delta)$ sufficiently large, so we can apply \cref{thm:trigBVerror} to bound the second factor by $O_A\Big(x^{1/2} ( \log x)^{-A}\Big)$ for any $A>0$. So the total error is $\ll_A y_{\max}^2 x ( \log x)^{-A}$. This gives the same shape of error as in \cite[(5.20)]{BoundedGaps}, so \eqref{eqn:Maynard S2} follows by the same arguments as in \cite{BoundedGaps}, using \cref{rem:o(1) is O(1/D0)} to replace $o(1)$ with $O(1/D_0)$. Also, the main term is adjusted by a factor of $\varphi(W)/\varphi(U)$. This factor equals $\varphi(\rad(N_1N_2))$ for $D_0$ sufficiently large.
\end{proof}

\begin{proof}[Proof of \cref{prop:Maynard ratio calculation}]
Recall that if we can show $S (x , \rho) > 0 $ for all large $x$, then there exist infinitely many integers $n$ such that at least $\lfloor \rho + 1 \rfloor $ of the $n+h_j$ are prime.

By definition of $\mathcal{M}_k$ and the density of smooth functions in the space of Riemann-integrable functions, there exists a smooth function $G_1$ such that $\sum_{j=1}^k J_{k}^{(m)} (G_1) > (\mathcal{M}_k - 2 \delta) I_k (G_1) > 0$. Applying \cref{prop:Maynard Prop S1 and S2} with $G = G_1$, we have that since $\log R / \log x = \theta / 2 - \delta $,
\begin{align*}
S (x, \rho) 
&= 
\frac{\varphi (W)^k x ( \log R)^k}{W^{k+1}}
\Big(
\frac{\varphi(N_1N_2)}{N_1N_2}
\widehat{F}_{I_1,I_2} (0,0) \frac{\log R}{\log x} \sum_{m=1}^k J_{k}^{(m)} (F)- \rho I_{k} (F) + o(1) \Big)
\\[1em]
&\geq  \frac{\varphi (W)^k x ( \log R)^k I_k (F)}{W^{k+1}} 
\Big( 
\frac{\varphi(N_1N_2)}{N_1N_2}
\widehat{F}_{I_1,I_2}(0,0) \Big( \frac{\theta}{2} - \delta \Big) ( \mathcal{M}_k - \delta) - \rho + o(1) \Big) .
\end{align*}
If we let
\[
\rho 
= 
\frac{\varphi(N_1N_2)}{N_1N_2}
\widehat{F}_{I_1,I_2}(0,0) \frac{\theta \mathcal{M}_k}{2} - \varepsilon,
\]
then by choosing $\delta = \delta(\varepsilon)$ small enough we see that $S(x,\rho)>0$ for all large $x$.
Since $\lfloor \rho +1 \rfloor = \big\lceil 
\frac{\varphi(N_1N_2)}{N_1N_2}
\widehat{F}_{I_1,I_2}(0,0) \theta M_k / 2
\big\rceil$ if $\varepsilon$ is small enough, we are done since at least $\lfloor \rho +1 \rfloor $ of the $n+h_j$ are prime for infinitely many $n$.
\end{proof}

\section{Proof of \cref{prop:Pintz is satisfied}}
\label{sec:Green Tao}

Recall that we are interested in the growth of the set $\mathcal{A}_{f_1,f_2,I_1,I_2}(\mathcal{H},\gamma_1)$ defined in \eqref{eqn:auxiliary set}.
We first restrict the sums $S_1$ and $S_2$ to the indices $n\in\N$ that lie in $\mathcal{R}(\mathcal{H},\gamma_1)$. That is, we define the sums
\begin{align*}
S_1^+(\gamma_1)
&:= 
\sum_{\substack{x< n\le 2x\\n\equiv u_0\pmod{U} \\ n\in\mathcal{R}(\mathcal{H},\gamma_1)}} w_n,\\
S_2^+(\gamma_1)
&:= 
\sum_{\substack{x< n\le 2x\\n\equiv u_0\pmod{U} \\ n\in\mathcal{R}(\mathcal{H},\gamma_1)}} 
\Big(\sum_{j = 1}^k \1_\mathbb{P} (n + h_j) F_{I_1,I_2}( \theta_1(n + h_j), \theta_2(n + h_j))\Big)w_n.
\end{align*}
We want to show for $\gamma_1>0$ sufficiently small that the sums $S_j^{+}(\gamma_1)$ differ negligibly from $S_j$.
That is, we want to study the differences
\[
S_1^{-}(\gamma_1) := S_1 - S_1^{+}(\gamma_1), \qquad S_2^{-}(\gamma_1) := S_2 - S_2^{+}(\gamma_1).
\]
In fact, we will further decompose $S_1^{-}(\gamma_1)$ into the sums
\[
S_{1,p}^{(j)} := 
\sum_{\substack{x<n\leq 2x \\ n \equiv u_0 \pmod{U} \\ p \mid n+h_j}} w_n, \qquad 1 \leq j \leq k, 
\qquad p \text{ prime}.
\]

We have the following two lemmata due to Vatwani and Wong \cite{VW}.

\begin{lemma}
\label{lem:bound-on-S1pj}
For any $j\in\{1,\ldots,k\}$ and any prime $p > D_0$ with $\frac{\log p}{\log R} < \varepsilon$, one has for $\varepsilon>0$ sufficiently small that
\[
S_{1,p}^{(j)}
 \ll 
\frac{(\log p)^2}{p(\log R)^2} \frac{x(\log R)^k}{U}.
\]
\end{lemma}
\begin{proof}
A proof is given by Vatwani and Wong \cite[Lemma 5.1]{VW} when $U$ is equal to
\[
\prod_{\substack{p \leq D_0 \\ p \, \nmid \, d_K}} p
\]
where $d_K$ is the discriminant of a number field $K$.
By replacing $d_K$ with $N_1N_2$, we are done.
\end{proof}

\begin{lemma}
\label{lem:bounds on S1-S2-}
For $\varepsilon > 0$ sufficiently small, there exists $\gamma_1 = \gamma_1(\varepsilon) >0$ such that for $x = x(\varepsilon)$ sufficiently large,
\[
S_1^{-}(\gamma_1)
 \leq 
\varepsilon \frac{x (\log R)^k}{U} \qquad\text{and}\qquad 
S_2^{-}(\gamma_1) 
 \leq 
\varepsilon \frac{x (\log R)^k}{U}.
\]
\end{lemma}

\begin{proof}
The following is based on \cite[Lemma 5.2]{VW}. Since $S_1^{-}(\gamma_1)$ includes only values of $n$ for which $n + h_j$ has small prime factors for each $i$, we have that  
\[
S_1^{-}(\gamma_1)
 \leq 
\sum_{j=1}^k \sum_{p \leq (2x)^{\gamma_1}} S_{1,p}^{(j)}.
\]

Then, we have that, for our indices $p$ (and $x\geq 2$), $\frac{\log p}{\log R} \leq \gamma_1 \frac{4}{\theta-2\delta}$. Therefore, taking $\gamma_1 \leq \frac{\theta-2\delta}{8}\varepsilon$ and applying \cref{lem:bound-on-S1pj}, there exists a constant $c_3>0$ such that
\[
S_1^{-}(\gamma_1)
 \leq 
c_3
\frac{x (\log R)^k}{U} \sum_{p \leq (2x)^{\gamma_1}} \frac{(\log p)^2}{p(\log R)^2}.
\]
By the prime number theorem, one has that for $x = x(\gamma_1)$ sufficiently large,
\[
\sum_{p \leq (2x)^{\gamma_1}} \frac{(\log p)^2}{p}
 \ll 
\gamma_1^2 (\log x)^2.
\]
Therefore, noting that $\frac{\log x}{\log R} \ll 1$, there exists a constant $c_4>0$ such that
\[
S_1^{-}(\gamma_1)
 \leq 
c_4 \gamma_1^2 \frac{x (\log R)^k}{U}.
\]
We now choose $\gamma_1 = \gamma_1(\varepsilon)$ sufficiently small so that $c_4 \gamma_1^2 \leq \varepsilon$ and $\gamma_1 \leq \frac{\theta-2\delta}{8}\varepsilon$.
The result follows.
\end{proof}

We are now ready to prove \cref{prop:Pintz is satisfied}.

\begin{proof}[Proof of \cref{prop:Pintz is satisfied}]

To prove an estimate for the growth of the set $\mathcal{B} := \mathcal{A}_{f_1,f_2,I_1,I_2}(\mathcal{H},\gamma_1)$, we begin performing manipulations on the dyadic sum
\[
\sum_{\substack{x < n \leq 2x \\ n\in \mathcal{B}}} 1
\]
that are analogous to \cite[Proof of Theorem 5.5]{VW}.
We first notice that for $n\in\mathcal{B}$ and $\rho$ as in the proof of \cref{prop:Maynard ratio calculation},
\[
0  <  
\sum_{j=1}^k \1_{\mathbb{P}}(n+h_j) F_{I_1,I_2} (\cos \theta_1(n+h_j), \cos \theta_2(n+h_j)) - \rho
 \leq 
k-m
 \leq k.
\]
In particular, there exists a constant $c_5>0$ such that
\[
\sum_{\substack{x < n \leq 2x \\ n\in \mathcal{B}}} 1
 \geq 
c_5 \sum_{\substack{x < n\leq 2x \\ n \in \mathcal{B}}} \Big( 
\sum_{j=1}^k \1_{\mathbb{P}}(n+h_j) F_{I_1,I_2} (\cos \theta_1(n+h_j), \cos \theta_2(n+h_j)) - \rho
\Big).
\]
Notice that for $n\in\mathcal{R}(\mathcal{H},\gamma_1)$,
\begin{equation}
\label{eqn:divisor_sum_bound}
\sum_{d_j \mid n+h_j \, \forall j} 1
 \ll_{\gamma_1}  1.
\end{equation}
Also, from \cite[(5.9) and (6.3)]{BoundedGaps}, we have that
\begin{equation}
\label{eqn:lambda_max_bound}
\lambda_{\max} \ll y_{\max}(\log R)^k \ll (\log R)^k.
\end{equation}
It follows from \cref{eqn:divisor_sum_bound} and \cref{eqn:lambda_max_bound} that there exists a constant $c_{6} = c_{6}(\gamma_1) > 0$ such that for all $k$-tuples $d_1,\ldots,d_k$,
\[
c_5
 \geq 
\frac{c_{6}}{(\log R)^{2k}} \Big(\sum_{d_j \mid n+h_j \, \forall j} \lambda_{d_1,\ldots,d_k}\Big)^2
= 
\frac{c_{6}}{(\log R)^{2k}} w_n.
\]
Therefore, we have that
\begin{align*}
\sum_{\substack{x < n \leq 2x \\ n\in \mathcal{B}}} 1
 &\geq  
\frac{c_{6}}{(\log R)^{2k}} 
\sum_{\substack{x < n\leq 2x \\ n \in \mathcal{B}}} \Big( 
\sum_{j=1}^k \1_{\mathbb{P}}(n+h_j) F_{I_1,I_2} (\cos \theta_1(n+h_j), \cos \theta_2(n+h_j)) - \rho
\Big) w_n
\\
&=
\frac{c_{6}}{(\log R)^{2k}} 
\Big(
S_2^+(\gamma_1) - \rho S_1^{+}(\gamma_1)
\Big).
\end{align*}
By \cref{prop:Maynard Prop S1 and S2}, if $D_0$ is sufficiently large, then
\[
S_2 - \rho S_1 = 
x
\Big(1+O\Big(\frac{1}{D_0}\Big)\Big)
\frac{\varphi(W)^k(\log R)^k}{W^{k+1}}
\Big(
\sum_{m=1}^{k} J_k^{(m)}(G) - \rho \frac{\log R}{\log x} I_k(G)
\Big).
\]
Choose $m$ so that \eqref{eqn:m versus k} holds. By following \cref{prop:Maynard ratio calculation,thm:more explicit Maynard Tao} with $D_0$ sufficiently large but fixed with respect to $x$, there exists a constant $c_{7} > 0$ such that
\[
S_2 - \rho S_1 
 \geq 
2c_{7} x (\log x)^k.
\]
Applying \cref{lem:bounds on S1-S2-} with $\varepsilon$ sufficiently small so that $S_2^{-}(\gamma_1) + \rho S_1^{-}(\gamma_1) \leq c_{7} x(\log x)^k$, it follows that
\[
S_2^{+}(\gamma_1) - \rho S_1^{+}(\gamma_1) 
 \geq 
c_{7} x (\log x)^k.
\]
Taking $\gamma_2 = c_{6}c_{7}(\frac{\log x}{\log R})^{2k}$ and applying \cref{thm:Pintz} to the set $\mathcal{B}$ with constants $\gamma_1$ and $\gamma_2$, this completes the proof.
\end{proof}

\section{Automorphic $L$-functions}
\label{sec:Automorphic L-functions}

In this section, we recall some of the theory of automorphic representations and their associated $L$-functions.

\subsection{Automorphic representations}
Let $\A_{\Q}$ be the adele ring over $\Q$. For every natural number $m\in\N$, let $\mathfrak{F}_m$ be the set of cuspidal automorphic representations of $\GL_m(\A_{\Q})$ with trivial central character.  Let $\pi\in\mathfrak{F}_m$.  There exist smooth admissible representations $\pi_{\infty}$ of $\GL_{m}(\R)$ and $\pi_p$ of $\GL_m(\Q_p)$ for every prime $p$ such that $\pi$ equals the restricted tensor product
\begin{equation*}
\pi = \pi_{\infty}\otimes\Big(\sideset{}{'}{\bigotimes}_{p} \pi_p\Big).
\end{equation*}
Moreover, $\pi_p$ is ramified for at most finitely many $p$.  For every prime $p$, there exist $m$ complex numbers $(\alpha_{j,\pi}(p))_{j=1}^m$ that form the local $L$-function
\begin{equation*}
L(s,\pi_p) := \prod_{j=1}^m (1-\alpha_{j,\pi}(p)p^{-s})^{-1}.
\end{equation*}
In particular, the $L$-function associated to $\pi$ has an absolutely convergent Euler product
\begin{equation*}
\label{eqn:ordinary automorphic L-function}
L(s,\pi) := \prod_p L(s,\pi_p) =  
\prod_p \prod_{j=1}^m (1-\alpha_{j,\pi}(p)p^{-s})^{-1} =  \sum_{n=1}^\infty \lambda_{\pi}(n) n^{-s}, \qquad \Re(s)>1.
\end{equation*}
There exist $m$ complex Langlands parameters $(\kappa_{j,\pi})_{j=1}^m$ that form the local $L$-function
\begin{equation*}
L(s,\pi_{\infty}) :=  \pi^{-ms/2} \prod_{j=1}^m \Gamma\Big( \frac{s+\kappa_{j,\pi}}{2} \Big).
\end{equation*}
There also exists a conductor $q_{\pi}\in\N$, and $r_{\pi}\in\{0,1\}$ which is $1$ if and only if $\pi$ is the trivial representation, such that the completed $L$-function
\begin{equation*}
\Lambda(s,\pi) := (s(1-s))^{r_{\pi}} q_{\pi}^{s/2} L(s,\pi_{\infty}) L(s,\pi)
\end{equation*}
is entire of order $1$.  If $\widetilde{\pi}\in\mathfrak{F}_m$ is the dual of $\pi$, then
\begin{equation*}
\{ \alpha_{j,\widetilde{\pi}}(p)  \} = \{ \overline{\alpha_{j,\pi}(p)} \} , \qquad
\{ \kappa_{j,\widetilde{\pi}} \}     =  \{ \overline{\kappa_{j,\pi}} \}   , \qquad 
q_{\widetilde{\pi}} = q_{\pi}, \qquad 
r_{\widetilde{\pi}} = r_{\pi}.
\end{equation*}
There exists $W(\pi)\in\C$ of modulus $1$ such that the functional equation
\begin{equation*}
\Lambda(s,\pi) = W(\pi) \Lambda(1-s,\widetilde{\pi})
\end{equation*}
holds for all $s\in\C$.
The nontrivial zeros of $L(s,\pi)$ are the zeros of $\Lambda(s,\pi)$, and they lie in the critical strip $0 \leq \Re(s) \leq 1$.
Finally, the analytic conductor is defined as
\begin{equation*}
C(s,\pi) := 
q_{\pi} \prod_{j=1}^m (|s+\kappa_{j,\pi}| + 3),
\qquad
C(\pi) := C(0,\pi).
\end{equation*}

\subsection{Rankin--Selberg $L$-functions}
\label{subsection:rankin-selberg}
Let $\pi \in \mathfrak{F}_m$ and $\pi' \in \mathfrak{F}_{m'}$.  Let $r_{\pi\times\pi'}=1$ if $\pi'=\widetilde{\pi}$, and otherwise $r_{\pi\times\pi'}=0$. 
 For every prime $p \nmid q_{\pi} q_{\pi'}$, we define the local $L$-function
\[
L(s,\pi_p\times\pi_p') := 
\prod_{j = 1}^{m} \prod_{j'=1}^{m'} (1 - \alpha_{j,\pi}(p) \alpha_{j',\pi'}(p)p^{-s})^{-1}.
\]
For each $p \mid q_{\pi} q_{\pi'}$ and each pair of integers $(j,j')\in[1,m]\times[1,m']$,  Jacquet, Piatetskii-Shapiro, and Shalika \cite{JacquetPiatetskiiShapiroShalika} defined complex numbers $\alpha_{j,j,',\pi\times\pi'}(p)$ and $\kappa_{j,j',\pi\times\pi'}$ such that if
\[
L(s,\pi_p\times\pi_p') := \prod_{j=1}^{m} \prod_{j'=1}^{m'} (1-\alpha_{j,j',\pi, \pi '}(p) p^{-s})^{-1}
\]
and
\[
L(s,\pi_{\infty}\times\pi_{\infty}') := \pi^{-mm's/2} \prod_{j,j'} \Gamma \Big( \frac{s + \kappa_{j,j',\pi \times \pi'}}{2} \Big),
\]
then the Rankin--Selberg $L$-function
\begin{equation*}
L(s,\pi\times\pi') := 
\prod_p L(s,\pi_p\times\pi_p') = 
\sum_{n=1}^{\infty} \lambda_{\pi\times\pi'}(n) n^{-s}
\end{equation*}
converges absolutely for $\Re(s)>1$ and satisfies the following theorem.
\begin{theorem}
\label{thm:Jacquet Piatetskii-Shapiro Shalika}
Let $(\pi,\pi')\in\mathfrak{F}_m\times\mathfrak{F}_{m'}$. There exists a conductor $q_{\pi\times\pi'} \in\N$ such that the completed $L$-function
\begin{equation}
\label{eqn:completed L-function}
\Lambda(s,\pi\times\pi') := 
(s(1-s))^{r_{\pi\times\pi'}} q_{\pi \times \pi '}^{s/2} L(s,\pi_{\infty}\times\pi_{\infty}') L(s,\pi\times\pi')
\end{equation}
is entire of order $1$. Moreover, there exists $W(\pi\times\pi')\in\C$ of modulus $1$ such that
\begin{equation*}
\Lambda(s,\pi\times\pi') = 
W(\pi\times\pi') \Lambda (1-s,\widetilde{\pi}\times\widetilde{\pi}').
\end{equation*}
\end{theorem}
\begin{proof}
See \cite{JacquetPiatetskiiShapiroShalika}.
\end{proof}

\begin{remark}
\label{rem:analytic continuation of rs L-function}
Moreover, \cite{JacquetPiatetskiiShapiroShalika} shows that $L(s,\pi\times\pi')$ is holomorphic away from a possible pole at $s=1$, which occurs if and only if $\pi' = \widetilde{\pi}$. 
\end{remark}

Let $\pi\in \mathfrak{F}_m$, $\pi'\in\mathfrak{F}_{m'}$, and $\chi$ be a primitive Dirichlet character modulo $q_{\chi}$.
The nontrivial zeros of $L(s,\pi\times(\pi'\otimes\chi))$ are the zeros of $\Lambda(s,\pi\times(\pi'\otimes\chi))$, and they lie in the critical strip $0 \leq \Re(s) \leq 1$.
For the conductor $q_{\pi\times(\pi'\otimes\chi)}$, Bushnell and Henniart \cite{BushnellHenniart}  proved that
\begin{equation}
\label{eqn:RS conductor divides}
q_{\pi\times(\pi'\otimes\chi)}
\mid  
q_{\pi}^{m'} q_{\pi'}^{m} q_{\chi}^{mm'}.
\end{equation}
Define
\[
C( s , \pi\times(\pi'\otimes\chi)) := q_{\pi\times(\pi '\otimes\chi)} \prod _{j=1} ^{m} \prod _{j' = 1} ^{m'} ( | s + \kappa _{j,j',\pi \times \pi '}| +3).
\quad
C(\pi\times(\pi'\otimes\chi)) :=
C(0,\pi\times(\pi'\otimes\chi)),
\]
Brumley \cite[Appendix]{HumphriesBrumley} proved that
\begin{equation}
\label{eqn:analytic conductor estimate}
C(\pi\times(\pi'\otimes\chi)) 
 \ll 
C(\pi)^{m'} C(\pi')^{m} q_{\chi}^{mm'},
\end{equation}
where the implied constant only depends on $m,m'$. One can also check that
\begin{equation}
\label{eqn:separating_analytic_conductor}
C( it, \pi \times (\pi '\otimes\chi)) \leq 
C( \pi\times(\pi'\otimes\chi)) (|t| +3)^{mm'}.
\end{equation}

If $\chi$ is a primitive Dirichlet character, we denote the following Dirichlet coefficients
\begin{align*}
L(s,\pi \times (\pi'\otimes\chi))  =  \sum_{n=1}^{\infty} \lambda_{\pi \times (\pi'\otimes\chi)}(n) n^{-s},\qquad 
\frac{1}{L(s,\pi \times (\pi'\otimes\chi))} =  \sum_{n=1}^{\infty} \mu_{\pi \times (\pi'\otimes\chi)}(n) n^{-s}, 
\end{align*}
\begin{align*}
-\frac{L'}{L}(s,\pi \times (\pi'\otimes\chi))  =  \sum_{n=1}^{\infty} \Lambda_{\pi \times (\pi'\otimes\chi)}(n) n^{-s}.
\end{align*}
Note, if $\chi$ is the trivial character then $L(s,\pi \times (\pi'\otimes\chi)) = L(s,\pi \times \pi ')$ and if $\pi ' $ is the trivial representation then $L(s, \pi \times (\pi '\otimes\chi)) = L(s , \pi)$, so we are setting notation for these $L$-functions as well. Using computations analogous to those from Luo, Rudnick, and Sarnak \cite[Lemma 2.1]{LuoRudnickSarnak}, we find for $ (q_{\pi} q_{\pi '} , q_{\chi}) = 1$ that $\alpha_{j,\pi\times(\pi'\otimes\chi)}(p) = \alpha_{j,\pi\times\pi'}(p) \chi(p)$. In particular, if $n\in\N$, then
\begin{align}
\lambda_{\pi\times(\pi'\otimes\chi)}(n) 
&= 
\lambda_{\pi\times\pi'}(n) \chi(n),
\nonumber \\ \label{eqn:splitting of Dirichlet coefficients}
\Lambda_{\pi\times(\pi'\otimes\chi)}(n) 
&= 
\Lambda_{\pi\times\pi'}(n) \chi(n),
\\ \nonumber 
\mu_{\pi\times(\pi'\otimes\chi)}(n) 
&= 
\mu_{\pi\times\pi'}(n) \chi(n).
\end{align}
We note that finding an upper bound on the absolute value of one of these numbers is independent of $\chi$.
Furthermore, there exist theorems which allow decoupling of the dependencies on $\pi$ and $\pi'$.
For example, Brumley \cite[Appendix]{SoundThornerBrumley} proved that for all $n\in\N$,
\begin{equation}
\label{eqn:big Lambda splits}
|\Lambda_{\pi\times\pi'}(n)|  \leq \sqrt{\Lambda_{\pi\times\widetilde{\pi}}(n) \Lambda_{\pi'\times\widetilde{\pi}'}(n)}.
\end{equation}
Also, Jiang, L\"{u}, and Wang \cite[Lemma 3.1]{JLW} later proved that for all $n\in\N$,
\begin{equation}
\label{eqn:little lambda splits}
|\lambda_{\pi\times\pi'}(n)| 
 \leq 
\sqrt{\lambda_{\pi\times\widetilde{\pi}}(n)  \lambda_{\pi'\times\widetilde{\pi}'}(n)}.
\end{equation}

\subsection{Isobaric sums}

For $t\in\R$, consider the character $\norm^{it}$, where $\norm$ is the adelic norm.  The $L$-function of $\pi\otimes\norm^{it}$ is $L(s+it,\pi)$.  If $\ell\in\N$ and $(\pi_j,t_j)\in\mathfrak{F}_{m_j}\times\R$ for all $1\leq j\leq \ell$, then the isobaric sum $\Pi = \pi_1\otimes\norm^{it_1}\boxplus \cdots \boxplus\pi_{\ell}\otimes\norm^{it_{\ell}}$ is an automorphic representation of $\GL_{m_1+\cdots+m_{\ell}}(\A_{\Q})$ whose $L$-function is
\[
L(s,\Pi) = \prod_{j=1}^{\ell} L(s+it_j,\pi_j).
\]
We let $\mathfrak{A}_m$ denote the set of isobaric automorphic representations of $\GL_m(\A_{\Q})$.

We extend the Rankin--Selberg convolution to isobaric sums of cuspidal automorphic representations. 
If $\Pi =  \pi_1\otimes\norm^{it_1}\boxplus \cdots \boxplus\pi_{\ell}\otimes\norm^{it_{\ell}} \in \mathfrak{A}_m$ and $\Pi ' = \pi '_1\otimes\norm^{it'_1}\boxplus \cdots \boxplus\pi '_{\ell '}\otimes\norm^{it'_{\ell '}}  \in \mathfrak{A}_{m'} $, then their Rankin--Selberg convolution is
\begin{equation}
    \label{eqn:rankin_isobaric_def}
L(s,\Pi\times\Pi') := \prod_{j=1}^{\ell} \prod_{j' =1}^{\ell '} L(s + it_j + it '_{j'},\pi_{j} \times\pi '_{j'}).
\end{equation}

\begin{lemma}
\label{lem:HoffsteinRamakrishnan}
Let $\Pi$ be an isobaric sum of cuspidal automorphic representations. Then the Dirichlet series for $L(s,\Pi\times \widetilde{\Pi})$, $\log L(s,\Pi\times\widetilde{\Pi})$, and $-\frac{L'}{L}(s,\Pi\times\widetilde{\Pi})$ with $\Re(s)>1$ have nonnegative Dirichlet coefficients.
\end{lemma}
\begin{proof}
See \cite[Lemma a]{HoffsteinRamakrishnan}.
\end{proof}

\subsection{$L$-function Estimates in the Critical Strip}
\label{sec:L-function estimates}

Given $\pi \in \mathfrak{F}_m$, $\pi' \in \mathfrak{F}_{m'} $, and $\chi$ a primitive Dirichlet character modulo $q_{\chi}$, we prove estimates for the auxiliary function
\[
\mathcal{L} (s, \pi\times(\pi'\otimes\chi))
:= 
\Big( \frac{s-1}{s+1} \Big)^{r_{\pi \times (\pi '\otimes\chi)}} L(s,\pi\times(\pi'\otimes\chi))
\]
and its derivative in the region $\Re(s) \geq 0$.

\begin{lemma}\label{lem:Lfunctionverticalestimates}
For all $\sigma \geq 0$, $t \in \RR$, $j \in \{ 0 ,1 \} $, and $\varepsilon > 0$, 
\[
\mathcal{L}^{(j)} (\sigma+it,\pi\times(\pi'\otimes\chi))
 \ll_{\pi,\pi',\varepsilon}  
\big(q_{\chi} (|t|+3)\big)^{mm' \max (1-\sigma  ,0) /2 + \varepsilon}.
\]
Moreover, if $r_{\pi \times (\pi '\otimes\chi)} = 0$ or $|\sigma+it-1| > 1/4$, then
\begin{equation}
\label{eqn:Lverticalestimate}
L^{(j)} (\sigma+it,\pi\times(\pi'\otimes\chi))
\ll_{\pi,\pi',\varepsilon}  
\big(q_{\chi} (|t|+3)\big)^{mm' \max (1-\sigma  ,0) /2 + \varepsilon}.
\end{equation}
\end{lemma}

\begin{proof}
See \cite[Lemma 3.2]{HarcosThorner}.
\end{proof}

\section{Symmetric Power $L$-functions}
\label{sec: symmetric powers}

\subsection{Automorphy of symmetric powers}

Let $f\in S_k^{\mathrm{new}}(\Gamma_0(N))$ be a newform as in \cref{sec:introduction}. For each $m\in\N$, there exists a local factor at the primes $p\nmid N$ of the form
\begin{equation}
\label{eqn:m-th symmetric power unramified local factor}
L_p(s,\Sym^mf) := 
\prod_{j=0}^m (1 - \alpha_{j,\Sym^m f}(p)p^{-s})^{-1}, \qquad
\alpha_{j,\Sym^m f}(p) := e^{i(m-2j)\theta_{f}(p)}.
\end{equation}
For each prime $p\mid N$, there exists a local factor $L_p(s,\Sym^m f)$ such that the $m$-th symmetric power $L$-function of $f$ is given by
\begin{equation*}
L(s,\Sym^m f) := 
\prod_{p \mid N} L_p(s,\Sym^{m}f)
\prod_{p \nmid N} \prod_{j = 0}^m (1 - e^{i(m-2j)\theta_f(p)}p^{-s})^{-1} = \prod_p \prod_{j=0}^m (1-\alpha_{j,\Sym^m f}(p)p^{-s})^{-1},
\end{equation*}
which converges absolutely for $\Re(s)>1$.

\begin{theorem}[Newton and Thorne, \cite{NewtonThorne1,NewtonThorne2}]
\label{thm:NewtonThorne}
Let $f$ be a non-CM holomorphic newform with trivial nebentypus. Then, there exists $\pi\in\mathfrak{F}_{m+1}$ such that $L(s,\Sym^{m} f) = L(s,\pi)$.
\end{theorem}

For non-CM holomorphic newforms $f_1$ and $f_2 $, recall that the notation $f_1 \not \sim f_2 $ means that if $\psi $ is a primitive Dirichlet character, then $f_1 \neq f_2\otimes\psi $. Since the nebentypus characters of $f_1 $ and $f_2 $ are trivial, it follows from \cref{thm:NewtonThorne} that there exist self-dual representations $\pi \in \mathfrak{F}_{m_1+1}$ and $\pi' \in \mathfrak{F}_{m_2+1}$ such that
\begin{equation*}
L(s,\Sym^{m_1}f_1\times(\Sym^{m_2}f_2\otimes\chi)) = 
L(s,\pi\times(\pi'\otimes\chi)).
\end{equation*}
In order to determine whether $L(s,\Sym^{m_1}f_1\times(\Sym^{m_2}f_2\otimes\chi))$ has a pole at $s=1$ or not, we have the following proposition which follows from the work of Rajan \cite[Corollary 5.1]{Rajan}.
\begin{proposition}
\label{prop:symmetric power pair L-funcs are entire}
If $f_1 \not \sim f_2 $ and $(m_1,m_2) \neq (0,0)$, then $L(s,\Sym^{m_1}f_1\times(\Sym^{m_2}f_2\otimes\chi))$ is entire.
\end{proposition}

\begin{proof}
By \cref{thm:NewtonThorne}, let $\pi,\pi'$ be the cuspidal automorphic representations associated to $\Sym^{m_1}f_1$ and $\Sym^{m_2}f_2$ respectively.  Notice that $\pi$ and $\pi'$ are self-dual because our newforms have trivial nebentypus. Then, by \cref{thm:Jacquet Piatetskii-Shapiro Shalika}, it suffices to show that $\pi \ncong \pi'\otimes\chi$.
If $m_1 \neq m_2$, then $\pi \ncong \pi'\otimes\chi$ because $\pi$ has degree $m_1+1$ and $\pi'\otimes\chi$ has degree $m_2+1$.
Therefore, we may assume without loss of generality that $m_1 = m_2 = \ell > 0$.

Rajan \cite[Corollary 5.1]{Rajan} proved that if
\begin{equation}
\label{eqn:upper density trick}
\limsup_{x\to\infty}  \frac{\#\{p\leq x: U_{\ell}(\cos \theta_1(p)) = U_{\ell}(\cos \theta_2(p))\}}{\pi(x)} 
 >  0,
\end{equation}
then there exists a Dirichlet character $\psi$ such that $f_1 = f_2\otimes\psi$.
Thus, assume for the sake of contradiction that $\pi = \pi'\otimes\chi$.
This implies that for all but finitely many primes $p$, one has
\begin{equation*}
U_{\ell}(\cos \theta_1(p))
= 
U_{\ell}(\cos \theta_2(p)) \chi(p).
\end{equation*}
Therefore, \eqref{eqn:upper density trick} holds because the (natural) density of primes $p$ such that $\chi(p) = 1$ equals $1/\varphi(q_{\chi})$. 
This implies that $f_1$ and $f_2$ are related by twist, which is a contradiction.
The proposition follows.
\end{proof}

The following lemma will be useful later.

\begin{lemma}
\label{lem:sym m1 f x sym m2 f splits}
Let $f$ be a holomorphic newform with level $N$ and trivial nebentypus. For all $m_1,m_2\in \N$, one has
\begin{equation*}
L(s,\Sym^{m_1}f \times \Sym^{m_2}f) = 
\prod_{j=0}^{\min(m_1,m_2)} L(s,\Sym^{|m_1-m_2|+2j} f).
\end{equation*}
\end{lemma}

\begin{proof}
This follows from the Clebsch--Gordan identities for ${\rm SU}(2)$.
\end{proof}

\subsection{Estimates on Dirichlet Coefficients}

Note that in applications \eqref{eqn:splitting of Dirichlet coefficients} holds, which allows us to assume for our purposes of bounding coefficients that $\chi$ is trivial. We say that an $L$-function satisfies the Generalized Ramanujan Conjecture (GRC) if its Satake parameters are bounded in modulus by $1$, and its Langlands parameters have nonnegative real parts. We have the following two lemmata.

\begin{lemma}
\label{lem:LsSym m f satisfies GRC}
Let $f$ be a holomorphic newform with level $N$ and trivial nebentypus. Then, for every $m\in\N\cup\{0\}$, $L(s,\Sym^m f)$ satisfies GRC.
\end{lemma}

\begin{proof}
For $m=0$, $L(s,\Sym^0 f) = \zeta(s)$ satisfies GRC. Therefore, let us assume $m\in\N$.  If $p$ is a prime not dividing $N$, then the Satake parameters of $L_p(s,\Sym^m f)$ are bounded in modulus by $1$ by considering \eqref{eqn:m-th symmetric power unramified local factor}.
To show that the other primes $p\mid N$ have local factors whose parameters are bounded in modulus by $1$, as well as showing that the Langlands parameters have nonnegative real parts, we proceed as in \cite[Proof of Corollary 7.1.15]{Big2023AutomorphyTeam}.
\end{proof}

\begin{lemma}
\label{lem:absolute Dirichlet coefficient estimates}
Let $f_1$ and $f_2$ be non-CM holomorphic newforms.
For all $m_1,m_2\in\N\cup\{0\}$ and $n\in\N$,
\begin{align}
\label{eqn:big Lambda estimate}
|\Lambda_{\Sym^{m_1} f_1 \times \Sym^{m_2} f_2}(n)| 
&\leq
(m_1+1)(m_2+1)\Lambda(n),
\\
\label{eqn:little lambda estimate}
|\lambda_{\Sym^{m_1} f_1 \times \Sym^{m_2} f_2}(n)| 
&\leq
d_{(m_1+1)(m_2+1)}(n),
\\
\label{eqn:mu estimate}
|\mu_{\Sym^{m_1} f_1 \times \Sym^{m_2} f_2}(n)| 
 &\ll_{m_1,m_2}  d_2(n)^{(m_1+1)(m_2+1)}.
\end{align}
\end{lemma}

\begin{proof}
We apply \eqref{eqn:little lambda splits} with $\pi = \Sym^{m_1}f_1$, $\pi' = \Sym^{m_2}f_2$ in order to obtain
\begin{equation*}
|\Lambda_{\pi\times\pi'}(n)|  \leq 
\sqrt{\Lambda_{\pi\times\pi}(n) \Lambda_{\pi'\times\pi'}(n)}.
\end{equation*}
When $n$ is not a prime power, $\Lambda_{\pi\times\pi'}(n) = 0$.
When $n = p^{\ell}$ is a prime power, \cref{lem:sym m1 f x sym m2 f splits} implies that
\begin{equation*}
\Lambda_{\pi\times\pi}(p^{\ell}) = 
\sum_{j=0}^{m_1} \Lambda_{\Sym^{2j}f_1}(p^{\ell})
= 
\sum_{j=0}^{m_1} \sum_{r = 0}^{2j} \alpha_{r,\Sym^{2j}f_1}(p)^{\ell} \log p,
\end{equation*}
and similarly for $\Sym^{m_2}f_2$.  The bound \eqref{eqn:big Lambda estimate} now follows from \cref{lem:LsSym m f satisfies GRC}.

For the proof of \eqref{eqn:little lambda estimate}, one relates $\lambda_{\pi\times\pi'}(n)$ to $\Lambda_{\pi\times\pi'}(n)$ by formally relating the Dirichlet series $L(s,\pi\times\pi')$ and $-\frac{L'}{L}(s,\pi\times\pi')$. We reduce to the case of prime powers, since $\lambda_{\pi\times\pi'}(n)$ is a multiplicative function. For $n = p^{\ell}$, applying formal integration and exponentiation to $-\frac{L'}{L}(s,\pi\times\pi')$ yields the relation between Dirichlet coefficients
\begin{equation}
\label{eqn:formal relation}
\lambda_{\pi\times\pi'}(p^{\ell}) = 
\sum_{r=1}^{\ell} \frac{1}{r!} \sum_{j_1 + \cdots + j_r = \ell} \frac{\Lambda_{\pi\times\pi'}(p^{j_1}) \cdots \Lambda_{\pi\times\pi'}(p^{j_r})}{j_1\cdots j_r (\log p)^r}, \qquad
j_i \in \N.
\end{equation}
Applying \eqref{eqn:big Lambda estimate}, we obtain \eqref{eqn:little lambda estimate} since
\begin{equation*}
|\lambda_{\pi\times\pi'}(p^{\ell})|  \leq 
\sum_{r=1}^{\ell}
\frac{[(m_1+1)(m_2+1)]^r}{r!}
\sum_{j_1+\cdots+j_r=\ell} \frac{1}{j_1\cdots j_r}
= 
d_{(m_1+1)(m_2+1)}(p^{\ell}).
\end{equation*}

Lastly, we prove \eqref{eqn:mu estimate}. For all prime powers $p^{\ell}$, we have that
\[
\mu_{\pi\times\pi'}(p^{\ell}) 
= 
(-1)^{\ell}
\sum_{j_1 < \cdots < j_{\ell}} \alpha_{j_1,\pi\times\pi'}(p) \cdots \alpha_{j_{\ell},\pi\times\pi'}(p), \qquad
1  \leq j_i  \leq (m_1+1)(m_2+1).
\]
In particular, $\mu_{\pi\times\pi'}(p^{\ell}) = 0$ for $\ell > (m_1+1)(m_2+1)$. Also, if $p\nmid N_1N_2$, then
\[
|\mu(p^{\ell})|  \leq \binom{(m_1+1)(m_2+1)}{\ell}  \leq (\ell+1)^{(m_1+1)(m_2+1)} = d_2(p^{\ell})^{(m_1+1)(m_2+1)}.
\]
Thus, because there are only finitely many $p \mid N_1N_2$, we deduce by the multiplicativity of $\mu_{\pi\times\pi'}(n)$ and $d_2(n)$ that \eqref{eqn:mu estimate} holds.
\end{proof}

We have the following lemma due to Norton \cite{Norton} that allows us to obtain summatory averages of products of divisor functions.
\begin{lemma}
\label{lem:Norton}
Let $x \geq 2$ and $j_1,\ldots,j_r \in \N$. Then, there exists a constant $c_{8} = c_{8}(j_1,\ldots,j_r) > 0$ such that
\[
\sum_{n\leq x} d_{j_1}(n) \cdots d_{j_r}(n) 
 \ll_{j_1,\ldots,j_r}  
x(\log x)^{c_{8}}
\]
\end{lemma}
\begin{proof}
See \cite[Theorem 1.7]{Norton}.
\end{proof}

\subsection{Sums of Dirichlet Coefficients}

We now address sums of the form
\[
\sum_{n \leq x} \lambda_{\Sym^{m_1}f_1 \times (\Sym^{m_2}f_2 '\otimes\chi)}(n).
\]
We have the following lemma that is sufficient for the regime of conductors used in Vaughan's method in \cref{sec:BV}.

\begin{lemma}
\label{lem:summatory Dirichlet estimate}
Let $\chi$ be a primitive Dirichlet character modulo $q_{\chi}$ with $(q_{\chi},N_1N_2) = 1$.
Let $\theta \in (0,\frac{2}{(m_1+1)(m_2+1)})$ and $c(\theta) = \frac{2-(m_1+1)(m_2+1)\theta}{(m_1+1)(m_2+1)+1}$.
We have that if $q_{\chi} \leq x^{\theta}$, then for all $\varepsilon>0$,
\[
\sum_{n\leq x} \lambda_{\Sym^{m_1}f_1\times(\Sym^{m_2}f_2\otimes\chi)}(n)
 \ll_{m_1,m_2,\theta,\varepsilon}  
x^{1 - c(\theta)+\varepsilon}
\]
holds uniformly over all such $\chi$.
\end{lemma}

\begin{proof}
By \cref{eqn:splitting of Dirichlet coefficients}, \cref{eqn:little lambda estimate}, and the fact that $d_r(n) \ll_{r,\varepsilon} n^{\varepsilon}$, one has that
\[
|\lambda_{\Sym^{m_1}f_1 \times (\Sym^{m_2}f_2\otimes\chi)}(n)| \leq
d_{(m_1+1)(m_2+1)}(n)
\ll_{m_1,m_2,\varepsilon} n^{\varepsilon}.
\]
Therefore, $L(s,\Sym^{m_1}f_1\times(\Sym^{m_2}f_2\otimes\chi))$ satisfies the hypotheses of  \cite[Proposition 1.1]{FriedlanderIwaniecSummation}.
Therefore, we have for $x \geq q_{\Sym^{m_1}f_1 \times (\Sym^{m_2}f_2\otimes\chi)}^{1/2}$ that
\[
\sum_{n\leq x} \lambda_{\Sym^{m_1}f_1\times(\Sym^{m_2}f_2\otimes\chi)}
\ll
q_{\chi}^{\frac{(m_1+1)(m_2+1)}{(m_1+1)(m_2+1)+1}} x^{\frac{(m_1+1)(m_2+1)-1}{(m_1+1)(m_2+1)+1} + \varepsilon}
\leq 
x^{1-c(\theta)+\varepsilon},
\]
where the implied constant only depends on $f_1,f_2,m_1,m_2,\varepsilon$, and the Langlands parameters of $L(s,\Sym^{m_1}f_1\times(\Sym^{m_2}f_2\otimes\chi))$.
One may observe that the Langlands parameters only depend on the weights $k_1,k_2$ of $f_1,f_2$ respectively and $\chi(-1) = \pm 1$.
Therefore, the implied constant only depends on $f_1,f_2,m_1,m_2,\varepsilon$.

In order to handle $x \leq q_{\Sym^{m_1}f_1 \times (\Sym^{m_2}f_2\otimes\chi)}^{1/2}$, we observe by \eqref{eqn:RS conductor divides} that
\[
q_{\Sym^{m_1}f_1 \times (\Sym^{m_2}f_2\otimes\chi)}^{1/2}
\ll_{m_1,m_2} q_{\chi}^{(m_1+1)(m_2+1)/2}
\leq
x^{\theta(m_1+1)(m_2+1)/2}.
\]
Since $\theta(m_1+1)(m_2+1)/2 < 1$, this range of $x$ is bounded in terms of $f_1,f_2,m_1,m_2,\theta$, and so by allowing the implied constant in our lemma to depend on $\theta$, we are done.
\end{proof}

\section{Bombieri--Vinogradov}
\label{sec:BV}

Let $\pi = \Sym^{m_1}f_1$ and $\pi' = \Sym^{m_2}f_2$ where $(m_1,m_2) \neq (0,0)$. 
For each $q\geq 1$ and $a\in(\Z/q\Z)^{\times}$, we define
\[
\psi_{\pi\times\pi'}(x;q,a) := 
\sum_{\substack{n\leq x \\ n \equiv a \Mod{q}}} \Lambda_{\pi\times\pi'}(n).
\]
Let $m=m_1+1$ and $m' = m_2+1$. In order to prove \cref{thm:trigBVerror}, our goal is to prove the following theorem.
\begin{theorem}
\label{thm:BV Main Thm}
For any level of distribution $\theta\in(0,\frac{2}{2mm'+1})$, one has for any $A>0$ that
\begin{equation}
\label{eqn:BV Main Thm}
\sum_{\substack{q \leq x^{\theta} \\ (q,q_{\pi}q_{\pi'}) = 1}} \max_{\substack{y \leq x \\ (a,q) = 1}} |\psi_{\pi\times\pi'}(y;q,a)|
 \ll \frac{x}{(\log x)^{A}},
\end{equation}
where the implied constant only depends on $\pi,\pi',\theta,A$.
\end{theorem}

First, we handle the contribution from moduli $q$ no larger than a fixed power of $\log x$.
The following result is proven by the recent work of Harcos and Thorner \cite{HarcosThorner}.

\begin{theorem}
\label{thm:Siegel Walfisz}
If $A>0$ is fixed, then
\[
\sum_{\substack{q \leq (\log x)^A\\ (q,q_{\pi}q_{\pi'}) = 1}} \max_{\substack{y \leq x \\ (a,q) = 1}} |\psi_{\pi\times\pi'}(y;q,a)|
 \ll \frac{x}{(\log x)^{A}},
\]
where the implied constant only depends on $\pi,\pi',\theta,A$.
\end{theorem}
\begin{proof}
Per \cref{prop:symmetric power pair L-funcs are entire}, if $\chi\pmod{q_{\chi}}$ is primitive, and $\gcd(q_{\chi},q_{\pi}q_{\pi'})=1$, then $L(s,\pi\times(\pi'\otimes\chi))$ is entire.
Therefore, the result follows from \cite[Theorem 2.1]{HarcosThorner}.
\end{proof}

Throughout, $U,V\geq 1$ are parameters that we will take to be small powers of $x$, in particular we may assume $\log U \ll \log x$ and $\log V \ll \log x$. For $\theta \in (0,\frac{2}{mm'})$, define $c(\theta) := \frac{2-\theta mm'}{mm'+1}$. We begin by proving the following theorem for
\[
\psi_{\pi\times(\pi'\otimes\chi)}(x) := 
\sum_{n\leq x} \Lambda_{\pi\times(\pi'\otimes\chi)}(n).
\]

\begin{theorem}
\label{thm:BV Primitive Character Version}
Let $\theta \in (0,\frac{2}{mm'})$. Then, for all $x \geq 1$, $Q \in [1,x^{\theta}]$, there exists $c_{9} = c_{9}(m,m') > 0$ such that for any $\varepsilon \in (0,\frac{1}{2})$,
\begin{multline*}
\sum_{\substack{q\leq Q \\ (q,q_{\pi}q_{\pi'}) = 1}} \frac{q}{\varphi(q)} \sum_{\substack{\chi\pmod{q} \\ \textup{$\chi$ primitive}}} \max_{y\leq x} |\psi_{\pi\times(\pi'\otimes\chi)}(y)|
\\
\ll_{\pi,\pi',\theta,\varepsilon}  
(Q^2x^{1/2} + x + Qx^{1-\varepsilon/2} +Qx^{1/2+\varepsilon} + Q^2 x^{1-c(\theta)+\varepsilon}) (\log x)^{c_{9}}.
\end{multline*}
\end{theorem}

In order to prove \cref{thm:BV Primitive Character Version}, we employ an analogue of Vaughan's identity, namely
\[
\Lambda_{\pi\times(\pi'\otimes\chi)}(n)
= 
\mathfrak{a}_1(n) + \mathfrak{a}_2(n) + \mathfrak{a}_3(n) + \mathfrak{a}_4(n),
\]
where
\begin{align*}
\mathfrak{a}_1(n) 
&:= 
\Lambda_{\pi\times(\pi'\otimes\chi)}(n) \1_{n\leq U},
\nonumber \\[1em]
\mathfrak{a}_2(n) 
&:= 
- \sum_{\substack{j\ell r = n \\ j \leq U \\ \ell \leq V}} \Lambda_{\pi\times(\pi'\otimes\chi)}(j) \mu_{\pi\times(\pi'\otimes\chi)}(\ell) \lambda_{\pi\times(\pi'\otimes\chi)}(r),
\nonumber \\[1em]
\mathfrak{a}_3(n) 
&:= 
\sum_{\substack{h\ell = n \\ \ell \leq V}} \mu_{\pi\times(\pi'\otimes\chi)}(\ell) \lambda_{\pi\times(\pi'\otimes\chi)}(h) \log(h),
\nonumber \\[1em]
\mathfrak{a}_4(n) 
&:= 
-\sum_{\substack{jk = n \\ k > 1 \\ j > U}} \Lambda_{\pi\times(\pi'\otimes\chi)}(j) \Big(\sum_{\substack{\ell r = k \\ \ell\leq V}} \lambda_{\pi\times(\pi'\otimes\chi)}(r) \mu_{\pi\times(\pi'\otimes\chi)}(\ell)\Big).
\end{align*}
Therefore, it suffices to estimate the average over primitive characters of the following sums.
\begin{align*}
S_j  &:=  
\sum_{n\leq y} \mathfrak{a}_{j}(n), \\
\Sigma_j  &:=  
\sum_{\substack{q \leq Q \\ (q,q_{\pi}q_{\pi}) = 1}} \frac{q}{\varphi(q)} \sum_{\substack{\chi\pmod{q} \\ \textup{$\chi$ primitive}}} \max_{y\leq x} |S_j|,
\qquad j\in\{1,2,3,4\}.
\end{align*}
Using \eqref{eqn:big Lambda estimate}, we estimate that
\begin{equation}
S_1 = \sum_{n \leq \min\{y,U\}} \Lambda_{\pi\times(\pi'\otimes\chi)}(n) 
 \ll_{\pi,\pi'}  U,\qquad  \Sigma_1  \ll_{\pi,\pi'}  Q^2U.
\label{eqn:Sigma1 estimate}
\end{equation}
To treat $S_2$, by definition
\begin{equation*}
S_2 = 
-\sum_{n\leq y} \sum_{\substack{j\ell r = n \\ j \leq U \\ \ell \leq V}} \Lambda_{\pi\times\pi'}(j) \mu_{\pi\times\pi'}(\ell) \lambda_{\pi\times\pi'}(r) \chi(n).
\end{equation*}
Putting $t = j\ell$ and swapping order of summation, we have that
\begin{equation*}
S_2 = 
-\sum_{t \leq UV} 
\Big(
\sum_{\substack{j \ell = t \\ j \leq U \\ \ell \leq V}} \Lambda_{\pi\times\pi'}(j) \mu_{\pi\times\pi'}(\ell)
\Big)
\sum_{r \leq y/t} \lambda_{\pi\times\pi'}(r) \chi(rt).
\end{equation*}
We split the sum over $t$ into the ranges $t \leq U$ and $U < t \leq UV$. We define
\begin{align*}
S_2' &= 
-\sum_{t \leq U} 
\Big(
\sum_{\substack{j \ell = t \\ j \leq U \\ \ell \leq V}} \Lambda_{\pi\times\pi'}(j) \mu_{\pi\times\pi'}(\ell)
\Big)
\sum_{r \leq y/t} \lambda_{\pi\times\pi'}(r) \chi(rt),
\\[1em]
S_2'' &= 
-\sum_{U < t \leq UV} 
\Big(
\sum_{\substack{j \ell = t \\ j \leq U \\ \ell \leq V}} \Lambda_{\pi\times\pi'}(j) \mu_{\pi\times\pi'}(\ell)
\Big)
\sum_{r \leq y/t} \lambda_{\pi\times\pi'}(r) \chi(rt),
\\[1em]
\Sigma_2'  &:=  
\sum_{\substack{q \leq Q \\ (q,q_{\pi}q_{\pi}) = 1}} \frac{q}{\varphi(q)} \sum_{\substack{\chi\pmod{q} \\ \textup{$\chi$ primitive}}} \max_{y\leq x} |S_2'|,
\\[1em]
\Sigma_2''  &:=  
\sum_{\substack{q \leq Q \\ (q,q_{\pi}q_{\pi}) = 1}} \frac{q}{\varphi(q)} \sum_{\substack{\chi\pmod{q} \\ \textup{$\chi$ primitive}}} \max_{y\leq x} |S_2''|.
\end{align*}
Notice that $S_2 = S_2'  + S_2''$ and $\Sigma_2 \leq \Sigma_2' + \Sigma_2''$.

\subsection{Large Sieve Terms}

We handle the terms $S_2''$ and $S_4$ using the following bilinear large sieve inequality derived in Davenport \cite{Davenport}.

\begin{proposition}
\label{prop:large sieve inequality}
Let $\{a_j\}, \{b_k\}$ be sequences of complex numbers. If $Q,J,K \geq 1$, then
\begin{multline*}
\sum_{q\leq Q} \frac{q}{\varphi(q)} \sum_{\substack{\chi\pmod{q} \\ \textup{$\chi$ primitive}}} \max_{u \leq JK} \Big| \underset{jk \leq u}{\sum_{1 \leq j \leq J} \sum_{1 \leq k \leq K}} a_j b_k \chi(jk) \Big|
\\
\ll  (Q^2+J)^{1/2} (Q^2+K)^{1/2}
\Big( \sum_{1\leq j \leq J} |a_j|^2 \Big)^{1/2}
\Big( \sum_{1\leq k \leq K} |b_k|^2 \Big)^{1/2}
\log(2JK),
\end{multline*}
where the implied constant is absolute and effective.
\end{proposition}
\begin{proof}
See \cite[Chapter 28]{Davenport}.
\end{proof}

To treat $S_4$, we use the fact that if $q_{\chi}$ is coprime to $q_{\pi} q_{\pi'}$, we have \eqref{eqn:splitting of Dirichlet coefficients}.
Therefore, we may write
\begin{equation*}
S_4 = 
\sum_{U < j \leq y/V} \Lambda_{\pi\times\pi'}(j) \sum_{V < k \leq y/j} \Big( \sum_{\substack{\ell r = k \\ \ell \leq V}} \lambda_{\pi\times\pi'}(r) \mu_{\pi\times\pi'}(\ell) \Big) \chi(jk).
\end{equation*}
We now consider a dyadic partition of the range $j$. For $1 \leq J \leq x$, Let
\begin{equation*}
S_4^{(J)} := 
\sum_{\substack{U < j \leq y/V \\ J < j \leq 2J}} \Lambda_{\pi\times\pi'}(j) \sum_{V < k \leq y/j} \Big( \sum_{\substack{\ell r = k \\ \ell \leq V}} \lambda_{\pi\times\pi'}(r) \mu_{\pi\times\pi'}(\ell) \Big) \chi(jk).
\end{equation*}
Applying \cref{prop:large sieve inequality} with $K = x/J$, we have that
\begin{multline*}
\sum_{\substack{q\leq Q \\ (q,q_{\pi}q_{\pi'}) = 1}} \frac{q}{\varphi(q)} \sum_{\substack{\chi\pmod{q} \\ \textup{$\chi$ primitive}}} \max_{y\leq x} |S_4^{(J)}| 
 \ll 
(Q^2+J)^{1/2} \Big(Q^2 + \frac{x}{J}\Big)^{1/2} 
\Big( \sum_{J \leq m \leq 2J} |\Lambda_{\pi\times\pi'}(j)|^2 \Big)^{1/2}
\nonumber\\ \times  
\Big( \sum_{k \leq x/J} \Big| \sum_{\substack{\ell r = k \\ \ell \leq V}} \lambda_{\pi\times\pi'}(r) \mu_{\pi\times\pi'}(\ell) \Big|^2 \Big)^{1/2} \log(x).
\end{multline*}
Using \cref{lem:absolute Dirichlet coefficient estimates,lem:Norton}, there exists $c_{10}  = c_{10}(m,m') >0$ such that
\begin{align*}
\sum_{\substack{q \leq Q \\ (q,q_{\pi}q_{\pi}) = 1}} \frac{q}{\varphi(q)} \sum_{\substack{\chi\pmod{q} \\ \textup{$\chi$ primitive}}} \max_{y\leq x} |S_4^{(J)}|
 &\ll_{\pi,\pi'}  
(Q^2+J)^{1/2} \Big(Q^2 + \frac{x}{J}\Big)^{1/2} (J\log J)^{1/2} x^{1/2}(\log x)^{c_{10}}
\nonumber \\ &\ll  
(Q^2x^{1/2} + QxJ^{-1/2} + Qx^{1/2} J^{1/2} + x) (\log x)^{c_{10}}.
\end{align*}
Summing dyadically over $[J,2J]$ so that the range $[U,x/V]$ is covered, this yields
\begin{equation}
\label{eqn:Sigma4 estimate}
\Sigma_4
 \ll_{\pi,\pi'}  
(Q^2x^{1/2} + QxU^{-1/2} + QxV^{-1/2} + x) (\log x)^{c_{10}}.
\end{equation}
The sum $S_2''$ is handled similarly to $S_4$.
We dyadically decompose $S_2''$ into the sums
\[
S_2^{(T)} := 
-\sum_{T \leq t \leq 2T} 
\Big(
\sum_{\substack{j\ell = t \\ j \leq U \\ \ell \leq V}} \Lambda_{\pi\times\pi'}(j) \mu_{\pi\times\pi'}(\ell)
\Big)
\sum_{r \leq y/t} \lambda_{\pi\times\pi'}(r) \chi(rt).
\]
Applying \cref{prop:large sieve inequality} with $J=T$ and $K=x/T$ yields
\begin{multline*}
\sum_{\substack{q \leq Q \\ (q,q_{\pi}q_{\pi}) = 1}} \frac{q}{\varphi(q)} \sum_{\substack{\chi\pmod{q} \\ \textup{$\chi$ primitive}}} \max_{y\leq x} |S_2^{(T)}| \ll 
(T + Q^2)^{1/2} \Big(\frac{x}{T} + Q^2\Big)^{1/2}
\Big(
\sum_{T \leq t \leq 2T} \Big| \sum_{\substack{j\ell = t \\ j\leq U \\ \ell \leq V}} \Lambda_{\pi\times\pi'}(j) \mu_{\pi\times\pi'}(\ell) \Big|^2
\Big)^{1/2} 
\\ \times  
\Big(
\sum_{r \leq x/T} \Big|
\lambda_{\pi\times\pi'}(r)
\Big|^2
\Big)^{1/2} \log(x).
\end{multline*}
Using \cref{lem:absolute Dirichlet coefficient estimates,lem:Norton}, there exists $c_{11} = c_{11}(m,m') > 0$ such that
\[
\sum_{\substack{q \leq Q \\ (q,q_{\pi}q_{\pi}) = 1}}
\frac{q}{\varphi(q)} \sum_{\substack{\chi\pmod{q} \\ \textup{$\chi$ primitive}}} \max_{y\leq x} \Big|S_2^{(T)}\Big|
 \ll_{\pi,\pi'}  
(Q^2x^{1/2} + QxT^{-1/2} + Qx^{1/2}T^{1/2} + x) (\log x)^{c_{11}},
\]
Summing dyadically over $[T,2T]$ so that the range $[U,UV]$ is covered, we obtain
\begin{equation}
\label{eqn:Sigma2'' estimate}
\Sigma_2''
 \ll_{\pi,\pi'}  
(Q^2x^{1/2} + QxU^{-1/2} + Qx^{1/2}U^{1/2}V^{1/2} + x) (\log x)^{c_{11}}.
\end{equation}

\subsection{Terms with sums of Dirichlet coefficients}

We handle the terms $S_2'$ and $S_3$ using \cref{lem:summatory Dirichlet estimate}.

To treat $S_2'$, we first have from \eqref{eqn:splitting of Dirichlet coefficients}, \cref{lem:absolute Dirichlet coefficient estimates}, and factoring out $|\chi(t)| \leq 1 $,
\begin{align}
|S_2'| &\leq
\sum_{t \leq U} 
\Big(\sum_{\substack{j\ell = t \\ \ell\leq V}} |\Lambda_{\pi\times\pi'}(j)\mu_{\pi\times\pi'}(\ell)| \Big)
\Big|
\sum_{r\leq y/t} \lambda_{\pi\times\pi'}(r) \chi(rt)
\Big|
\nonumber \\ &\ll_{\pi,\pi'}  
(\log U) \cdot \Big(\sum_{\ell\leq \min(U,V)} d_{2}(\ell)^{mm'}\Big) \cdot \sum_{t\leq U} \Big|
\sum_{r\leq y/t} \lambda_{\pi\times\pi'}(r) \chi(r)
\Big|  .
\end{align}
Applying \cref{lem:absolute Dirichlet coefficient estimates,lem:Norton}, there exists $c_{12} = c_{12}(m,m') >0$ such that
\begin{equation*}
\ll_{\pi,\pi'}  
\min(U,V) (\log x)^{c_{12}} \cdot \sum_{t\leq U} \Big|
\sum_{r\leq y/t} \lambda_{\pi\times\pi'}(r) \chi(r)
\Big| .
\end{equation*}
Thus, by \cref{lem:summatory Dirichlet estimate}, 
\[
|S_2'|  \ll_{\pi,\pi',\theta,\varepsilon}  
\min (U,V)  (\log x)^{c_{12}} \sum_{t\leq U} (y/t)^{1-c(\theta)+\varepsilon} \ll_{m,m',\theta,\varepsilon}  
\min (U,V)  x^{1-c(\theta)+\varepsilon} U^{c(\theta)} (\log x)^{c_{12}}.
\]

To treat $S_3$, we exchange the order of summation to obtain
\[
S_3 = \sum_{\ell \leq V} \mu_{\pi \times \pi '} (\ell) \sum_{h \leq y / \ell} \lambda_{\pi\times\pi'}(h) \log(h) \chi(\ell h).
\]
Summing by parts and factoring out $ \chi(\ell)$, we obtain 
\[
\sum_{h \leq y / \ell} \lambda_{\pi\times\pi'}(h) \log(h) \chi(\ell h) 
= 
\log (y/\ell) \chi (\ell) \sum_{h \leq y/ \ell} \lambda_{\pi \times \pi '} (h) \chi (h) - \chi (\ell) \int_{1}^{y/\ell}  \Big(\sum_{h\leq u} \lambda_{\pi\times\pi'}(h) \chi( h)\Big) \, \frac{du}{u}.
\]
Applying \cref{lem:summatory Dirichlet estimate} to the sums on the right, we see that
\begin{align*}
\sum_{h \leq y / \ell} \lambda_{\pi\times\pi'}(h) \log(h) \chi(\ell h) 
 &\ll_{\pi , \pi ' , \theta, \varepsilon}  
[\log (y/\ell)]  (y/\ell)^{1-c(\theta)+\varepsilon} + \int_1^{y/\ell} u^{1-c(\theta)+\varepsilon} \, \frac{du}{u}
\\ &\ll_{\pi , \pi ' , \theta, \varepsilon}  
(\log y/\ell) (y/\ell)^{1-c(\theta)+\varepsilon}.
\end{align*}
So, we have
\[
|S_3| 
 \ll_{\pi , \pi' , \theta, \varepsilon}  
(\log x) \sum_{\ell \leq V} | \mu_{\pi \times \pi '}  (\ell) | ( y/\ell)^{1-c(\theta)+\varepsilon} .
\]
Applying \cref{lem:absolute Dirichlet coefficient estimates,lem:Norton} and summing by parts, there exists $c_{13} = c_{13}(m,m') > 0$ such that
\begin{align*}
|S_3|
 &\ll_{\pi , \pi', \theta, \varepsilon}  
x^{1-c'} (\log x) \sum_{\ell \leq V} d_{2} (\ell)^{mm'} / \ell^{1-c(\theta)+\varepsilon}
\nonumber \\ &\ll_{\pi,\pi',\theta,\varepsilon}  
x^{1-c(\theta)+\varepsilon} V^{c(\theta)} (\log x)^{c_{13}}.
\end{align*}

Our bounds on $S_2'$ and $S_3$ yield the following upper bounds:
\begin{align}
\label{eqn:Sigma2' estimate}
\Sigma_2' 
&\ll_{\pi,\pi',\theta,\varepsilon}  
Q^2 x^{1-c(\theta)+\varepsilon} U^{1+c(\theta)} (\log x)^{c_{12}};
\\[1em] 
\label{eqn:Sigma3 estimate}
\Sigma_3
 &\ll_{\pi,\pi',\theta,\varepsilon}  
Q^2 x^{1-c(\theta)+\varepsilon} V^{c(\theta)} (\log x)^{c_{13}}.
\end{align}

\subsection{Conclusion of Bombieri--Vinogradov}

\begin{proof}[Proof of \cref{thm:BV Primitive Character Version}]

Piecing together \cref{eqn:Sigma1 estimate}--\cref{eqn:Sigma2'' estimate}, \cref{eqn:Sigma2' estimate}, and \cref{eqn:Sigma3 estimate},
there exists $c_{9} = c_{9}(m,m') > 0$ such that
\begin{multline*}
\sum_{\substack{q\leq Q \\ (q,q_{\pi}q_{\pi'}) = 1}} \frac{q}{\varphi(q)} \sum_{\substack{\chi\pmod{q} \\ \textup{$\chi$ primitive}}} \max_{y\leq x} |\psi_{\pi\times(\pi'\otimes\chi)}(y)|
\ll_{\pi,\pi',\theta,\varepsilon} \\
(Q^2x^{1/2} + x + QxU^{-1/2} + QxV^{-1/2} + U^{1/2} V^{1/2} Q x^{1/2} + Q^2x^{1-c(\theta)+\varepsilon} (U^{1+c(\theta)}+V^{c(\theta)})) (\log x)^{c_{9}}.
\end{multline*}
Setting $U = V = x^{\varepsilon}$, where $\varepsilon\in(0,\frac{1}{2})$, this yields \cref{thm:BV Primitive Character Version} by rescaling $\varepsilon$.
\end{proof}

\begin{proof}[Proof of \cref{thm:BV Main Thm}]

Let $Q=x^{\theta}$.
By character orthogonality, we may write
\[
\psi_{\pi\times\pi'}(y,\chi) := \sum_{n\leq y} \Lambda_{\pi\times\pi'}(n) \chi(n),\qquad \psi_{\pi\times\pi'}(y;q,a)=\frac{1}{\varphi(q)} \sum_{\chi\pmod{q}} \overline{\chi}(a) \psi_{\pi\times\pi'}(y,\chi).
\]
Note that the conductor $q$ is coprime to $q_{\pi} q_{\pi'}$ and $\chi$ is primitive, so we have that $\psi_{\pi\times\pi'}(y,\chi) = \psi_{\pi\times(\pi'\otimes\chi)}(y)$. Using the above expression for $\psi_{\pi\times\pi'}(y;q,a)$, we have that the left hand side of \eqref{eqn:BV Main Thm} is
\begin{equation*}
 \ll 
\sum_{\substack{q \leq Q \\ (q,q_{\pi}q_{\pi'}) = 1}} \frac{1}{\varphi(q)} \sum_{\chi\pmod{q}} \max_{y\leq x}|\psi_{\pi\times\pi'}(y,\chi)|.
\end{equation*}
We are able to convert between $\chi$ and the primitive character that induces it, call it $\chi_1$, at a negligible cost. Indeed,
\[
|\psi_{\pi\times\pi'}(y,\chi) - \psi_{\pi\times\pi'}(y,\chi_1)|\leq
\sum_{\substack{p^{\ell} \leq y \\ p\mid q_{\chi} q_{\pi} q_{\pi'}}} |\Lambda_{\pi\times\pi'}(p^{\ell})| \ll_{\pi,\pi'}  
\sum_{p\mid q_{\chi} q_{\pi} q_{\pi'}} \Big\lfloor\frac{\log y}{\log p}\Big\rfloor \log(p)\ll_{\pi,\pi'}  
(\log q_{\chi}y)^2.
\]
Originally, each character $\chi \pmod{q}$ is assigned a weighting $1/\varphi(q)$. When reducing to primitive characters, each primitive character $\chi \pmod{q}$ is assigned a weighting $\sum_{k\leq Q/q} \frac{1}{\varphi(kq)}$ that is $\ll \frac{\log Q}{\varphi(q)}$, since we may observe that $\varphi(kq) \geq \varphi(k)\varphi(q)$ and then apply an estimate of Mertens. Thus, the left hand side of $\eqref{eqn:BV Main Thm}$ is
\begin{equation*}
\ll_{\pi,\pi'}  
Q(\log Qx)^2 + 
(\log Q)\sum_{\substack{q \leq Q \\ (q,q_{\pi}q_{\pi'})=1}} \frac{1}{\varphi(q)} \sum_{\substack{\chi\pmod{q} \\ \textup{$\chi$ primitive}}} \max_{y\leq x} |\psi_{\pi\times(\pi'\otimes\chi)}(y)|.
\end{equation*}
The term $Q(\log Qx)^2 \ll_{\theta} x^{\theta}(\log x)^2$ is negligible since $\theta < 1$, and the factor of $\log Q$ in front of our primitive character sum is also negligible, since it is $\ll_{\theta} \log x$ and we are saving arbitrary log powers.
It therefore suffices to show that the sum over moduli is $\ll_{\pi,\pi',\theta,A} x(\log x)^{-A}$.
By a dyadic partition, we have that
\begin{equation*}
\sum_{\substack{q \leq Q \\ (q,q_{\pi}q_{\pi'})=1}} \frac{1}{\varphi(q)} \sum_{\substack{\chi\pmod{q} \\ \textup{$\chi$ primitive}}} \max_{y\leq x} |\psi_{\pi\times(\pi'\otimes\chi)}(y)|
 \ll 
\frac{(\log Q)}{R} \max_{R \leq Q} \sum_{R \leq q \leq 2R} \frac{q}{\varphi(q)} \sum_{\substack{\chi\pmod{q} \\ \textup{$\chi$ primitive}}} \max_{y\leq x} |\psi_{\pi\times(\pi'\otimes\chi)}(y)|.
\end{equation*}
If the maximum occurs at $R \leq (\log x)^{A'}$, where $A' = A'(m,m')>0$ will be chosen later, then we apply \cref{thm:Siegel Walfisz} in order to conclude the proof.

We now consider the case where the maximum occurs at $R \geq (\log x)^{A'}$.
By applying \cref{thm:BV Primitive Character Version}, we have that the left hand side of \eqref{eqn:BV Main Thm} is
\begin{align}
&\ll_{\pi,\pi',\theta,\varepsilon}  
\max_{(\log x)^{A'} \leq R \leq Q} (x^{1/2}R + xR^{-1} + x^{1-\varepsilon/2} + x^{1-c(\theta)+\varepsilon}R)(\log x)^{c_{9}}
\nonumber \\[1em] &\ll_{\varepsilon,A'}  
x(\log x)^{c_{9}-A'} + Qx^{1-c(\theta)+\varepsilon} (\log x)^{c_{9}}.
\end{align}
It therefore suffices to take $A'-c_{9} \geq A$ and $x^{1+\theta-c(\theta)+\varepsilon} = o(x)$. The latter condition says that
\begin{equation*}
\theta  <  
\frac{2}{2mm'+1} - \varepsilon \Big(\frac{mm'+1}{2mm'+1}\Big).
\end{equation*}
We now take $\varepsilon = \varepsilon(m,m',\theta)$ sufficiently small so that the above inequality holds for any choice of $\theta \in (0,\frac{2}{2mm'+1})$.
\end{proof}

We may now pass from \cref{thm:BV Main Thm}, which is counting prime powers with von Mangoldt weights, to counting primes. This is then used to prove \cref{thm:trigBVerror}.

\begin{theorem}
\label{thm:BV applied version}
Let $f_1 \not\sim f_2$ be non-CM holomorphic newforms with levels $N_1,N_2$ respectively. Let $m_1,m_2\in\N\cup\{0\}$ such that $(m_1,m_2) \neq (0,0)$. For any $\theta \in \Big(0,\frac{2}{2(m_1+1)(m_2+1)+1}\Big)$ and $A>0$,
\[
\sum_{\substack{q \leq x^{\theta} \\ (q,N_1N_2) = 1}}
\max_{(a,q)=1}
\Big|
\sum_{\substack{p \leq x \\ p \equiv a (q)}} U_{m_1}(\cos \theta_1(p)) U_{m_2}(\cos \theta_2(p))
\Big|
 \ll 
\frac{x}{(\log x)^A},
\]
where the implied constant only depends on $f_1,f_2,m_1,m_2,\theta,A$.
\end{theorem}
\begin{proof}
Let $\pi = \Sym^{m_1}f_1$ and $\pi' = \Sym^{m_2}f_2$.
\cref{thm:BV Main Thm} implies that for $0 < \theta < \frac{2}{2(m_1+1)(m_2+1)+1}$, one has for any $A>0$ that
\begin{equation*}
\sum_{\substack{q \leq x^{\theta} \\ (q,q_{\pi}q_{\pi'}) = 1}} \max_{\substack{y \leq x \\ (a,q) = 1}} |\psi_{\Sym^{m_1} f_1\times\Sym^{m_2} f_2}(y;q,a)|
 \ll 
\frac{x}{(\log x)^{A}}
\end{equation*}
where the implied constant only depends on $f_1,f_2, m_1, m_2,\theta,A$.
Using \eqref{eqn:RS conductor divides} and \eqref{eqn:big Lambda estimate}, a standard exercise in partial summation yields the result.
\end{proof}

\begin{proof}[Proof of \cref{thm:trigBVerror}]
We have
\begin{align*}
\sum_{\substack{x < p \leq 2x \\ p \equiv a \Mod{q}}} F  ( \theta_1(p) , \theta_2(p)) - \widehat{F}  (0,0) \frac{ \pi (2x) - \pi (x)}{\varphi (q)} 
& \\ =  
\widehat{F}  (0,0) \Big(\sum_{\substack{x < p \leq 2x \\ p \equiv a \Mod{q}}} 1 - \frac{\pi (2x) -  \pi (x)} {\varphi (q)} & \Big) \\
 + \sum_{\substack{0 \leq m_1 \leq M_1 \\ 0 \leq m_2 \leq M_2 \\ (m_1 , m_2) \neq (0,0)}}  \widehat{F}  (m_1 , m_2)&
\sum_{\substack{x < p \leq 2x \\ p \equiv a \Mod{q}}}U_{m_1} ( \cos \theta_1(p)) U_{m_2} ( \cos \theta_2(p)) .
\end{align*}
Taking the maximum over $(a,q) = 1 $, summing over $q \leq x^{\theta}$ with $(q,N_1N_2)=1$ and then applying the triangle inequality, we can use the classical Bombieri--Vinogradov theorem to bound the first term and \cref{thm:BV applied version} to bound the second term, completing the proof.
\end{proof}

\nocite{*}
\bibliographystyle{plain}
\bibliography{bib.bib}

\end{document}